\documentclass[12pt]{amsart}
\usepackage{amsmath,amssymb,amsthm,graphics,amscd,mathrsfs}
\usepackage{amscd, amsfonts, amssymb, amsthm}
\usepackage{fullpage, verbatim}
\usepackage[colorlinks, breaklinks, linkcolor=blue]{hyperref} 
\usepackage{breakurl}
\usepackage{array}
\usepackage{latexsym}
\usepackage{enumerate}
\newcommand\CA{{\mathscr A}}

\newcommand\FF{{\mathscr F}} 
\newcommand\CH{{\mathcal H}}
\newcommand\CS{{\mathcal S}} 
\newcommand\CIF{{\mathcal {IF}}} 
 
\newcommand\CIFAC{{\mathcal {IF\!AC}}}

\newcommand\CI{{\mathcal I}}  
\newcommand\CJ{{\mathcal J}}

\newcommand\CR{{\mathcal R}}

\newcommand\R{{\varrho}}
\newcommand\RR{{\mathscr R}} 
\newcommand\CW{{\mathcal W}}

\newcommand\BBK{{\mathbb K}}

\newcommand\BBR{{\mathbb R}}
\newcommand\BBZ{{\mathbb Z}}

\newcommand\codim{\operatorname{codim}}

\newcommand\Der{{\operatorname{Der}}}

\newcommand\Fix{{\operatorname{Fix}}}

\newcommand\pdeg{\operatorname{pdeg}}

\newcommand\rk{\operatorname{rk}}

\newcommand\hgt{\operatorname{ht}}

\newcommand\inverse{^{-1}}

\numberwithin{equation}{section}

\theoremstyle{plain}
\newtheorem{lemma}[equation]{Lemma}
\newtheorem{theorem}[equation]{Theorem}

\newtheorem{conjecture}[equation]{Conjecture}
\newtheorem{condition}[equation]{Condition}
\newtheorem{corollary}[equation]{Corollary}
\newtheorem{proposition}[equation]{Proposition}
\theoremstyle{definition}
\newtheorem{defn}[equation]{Definition}
\newtheorem{remark}[equation]{Remark}
\newtheorem{remarks}[equation]{Remarks}
\newtheorem{example}[equation]{Example}


\begin{document}

%%%%%%%%%%%%%%%%%%%%%%%%%%%%%%%%%%%%%%%%%%%%%%%%%%%%%%%%%%%%%%%%%%%%%%
%%%%%%%%%%%%% top matter stuff
%%%%%%%%%%%%%%%%%%%%%%%%%%%%%%%%%%%%%%%%%%%%%%%%%%%%%%%%%%%%%%%%%%%%%%
\title[Arrangements of ideal type]
{Arrangements of ideal type}

\author[G. R\"ohrle]{Gerhard R\"ohrle}
\address
{Fakult\"at f\"ur Mathematik,
Ruhr-Universit\"at Bochum,
D-44780 Bochum, Germany}
\email{gerhard.roehrle@rub.de}


\allowdisplaybreaks

\begin{abstract}
In 2006 Sommers and Tymoczko defined so called  
arrangements of ideal type $\CA_\CI$ stemming
from ideals $\CI$ in the set of positive 
roots of a reduced root system.
They showed in a case by case argument that 
$\CA_\CI$ is free if the root system is
of classical type or $G_2$ and
conjectured that this is also the case for 
all types.
This was established only very recently 
in a uniform manner 
by Abe, Barakat, Cuntz, Hoge and Terao.
The set of non-zero exponents of the free arrangement 
$\CA_\CI$ is given by the dual of the height
partition of the roots in the 
complement of $\CI$ in the set of positive roots,
generalizing the Shapiro-Steinberg-Kostant theorem which
asserts that  the 
dual of the height partition of the set of positive roots gives the 
exponents of the associated Weyl group.

Our first aim in this paper is to investigate a stronger
freeness property of the $\CA_\CI$.
We show that all $\CA_\CI$ are inductively free, 
with the possible exception of some cases in type $E_8$.

In the same paper from 2006, Sommers and Tymoczko 
define a Poincar\'e polynomial $\CI(t)$ associated with each ideal $\CI$
which generalizes the Poincar\'e polynomial $W(t)$
for the underlying Weyl group $W$.
Solomon showed that $W(t)$ 
satisfies a product 
decomposition depending on the exponents of $W$
for any Coxeter group $W$.
Sommers and Tymoczko showed in a case by case analysis 
in type $A_n$, $B_n$ and $C_n$,
and some small rank exceptional types that 
a similar factorization property holds for 
the Poincar\'e polynomials $\CI(t)$
generalizing the formula of Solomon for $W(t)$.
They conjectured that their multiplicative formula
for $\CI(t)$ holds in all types.
In our second aim to investigate this conjecture further,
the same inductive tools we develop to obtain inductive freeness
of the $\CA_\CI$ are also employed.
Here we also show that this conjecture holds 
inductively in almost all instances
with only a small number of possible exceptions.
\end{abstract}

\maketitle

\setcounter{tocdepth}{1}
\tableofcontents

%%%%%%%%%%%%%%%%%%%%%%%%%%%%%%%%%%%%%%%%%%%%%%%%%%%%%%%%%%%%%%%%%%%%%%
%%%%%%%%%%%%% article body...
%%%%%%%%%%%%%%%%%%%%%%%%%%%%%%%%%%%%%%%%%%%%%%%%%%%%%%%%%%%%%%%%%%%%%%

%%%%%%%%%%%%%%%%%%%%%%%%%%%%%%%%%%%%%%%%%%%%%%%%%%%%%%%%%%%%%%%%%%%%%%
%%%%%%%%%%%%% \S1 Introduction
%%%%%%%%%%%%%%%%%%%%%%%%%%%%%%%%%%%%%%%%%%%%%%%%%%%%%%%%%%%%%%%%%%%%%%
\section{Introduction}

Much of the motivation 
for the study of arrangements 
of hyperplanes comes 
from Coxeter arrangements. 
They consist of the reflecting hyperplanes 
associated with the 
reflections of the underlying Coxeter group.
While Coxeter arrangements are a well studied 
subject, 
subarrangements of the latter 
are considerably less well understood. 
In this paper we study certain arrangements which are 
associated with ideals in the set of positive roots of 
a reduced root system, so called 
\emph{arrangements of ideal type}  $\CA_\CI$,
Definition \ref{def:idealtype},
cf.~\cite[\S 11]{sommerstymoczko}.

\subsection{Ideals in $\Phi^+$}
\label{sect:idealexponents}
Let $\Phi$ be an irreducible, reduced root system
and let $\Phi^+$ be the set of positive roots 
with respect to some set of simple roots $\Pi$.
An \emph{(upper) order ideal}, 
or simply \emph{ideal} for short, 
of  $\Phi^+$, is a subset $\CI$  of $\Phi^+$ 
satisfying the following condition: 
if $\alpha \in \CI$ and $\beta \in \Phi^+$ so that 
$\alpha + \beta \in \Phi^+$, then $\alpha + \beta \in \CI$.

Recall the standard partial ordering 
$\preceq$ on $\Phi$: $\alpha \preceq \beta$
provided $\beta - \alpha$ is a $\BBZ_{\ge0}$-linear combination 
of positive roots, or $\beta = \alpha$. Then $\CI$ is an ideal in $\Phi^+$
if and only if whenever 
$\alpha \in \CI$ and $\beta \in \Phi^+$ so that 
$\alpha \preceq \beta$, then $\beta \in \CI$.

Let $\beta$ be in $\Phi^+$. Then $\beta = \sum_{\alpha \in \Pi} c_\alpha \alpha$
for $c_\alpha \in \BBZ_{\ge0}$.
The \emph{height} of $\beta$ is defined to be $\hgt(\beta) = \sum_{\alpha \in \Pi} c_\alpha$.
%Let $\theta$ be the highest root in $\Phi$.
%Then $h = \hgt(\theta) + 1$ is the 
%\emph{Coxeter number} of $W$.

Let $\CI \subseteq \Phi^+$ be an ideal and let 
\[
\CI^c := \Phi^+ \setminus \CI
\]
be its complement in 
$\Phi^+$. 
Further, define 
$\lambda_i :=|\{\alpha \in \CI^c \mid \hgt(\alpha) = i\}|$.
% to be the number of roots in $\CI^c$ of height $i$.
This gives the \emph{height partition} 
$\lambda_1 \ge \lambda_2 \ge \ldots $
of %the set of roots in 
$\CI^c$.
Let $s = \lambda_1$, the number of simple roots in $\CI^c$.

\begin{defn}
[{\cite[Def.~3.2]{sommerstymoczko}}]
\label{def:idealexp}
%Let $\CI \subseteq \Phi^+$ be an ideal.
The \emph{ideal exponents $m_i^\CI$} of an ideal $\CI$ in $\Phi^+$
are the parts of the dual of 
the height partition of $\CI^c$,
i.e.~$m_i^\CI := |\{\lambda_j \mid \lambda_j \ge s -i +1 \}|$, so that  
\begin{equation}
\label{eq:idealexp}
m_s^\CI \ge \ldots \ge m_1^\CI.
\end{equation}  
\end{defn}

Note that since $\lambda_1 > \lambda_2$, we have $m_1^\CI = 1$,
cf.~\cite[Prop.\ 3.1]{sommerstymoczko}.

This terminology is motivated as follows.
For $\CI = \varnothing$, we have $\CI^c = \Phi^+$.
A famous theorem asserts that 
the dual of the height partition of 
$\Phi^+$ gives the exponents of $W$.
This connection was first  discovered by Shapiro
(unpublished) and rediscovered independently
by Steinberg \cite[\S 9]{steinberg:reflectiongroups}. 
Kostant \cite{kostant:betti} was the first to provide 
a uniform proof.
Macdonald gave a proof using generating functions \cite{macdonald:coxeter}. 

\subsection{Arrangements of ideal type}
\label{sect:idealexponents}
Following \cite[\S 11]{sommerstymoczko}, 
we associate with an ideal $\CI$ in $\Phi^+$ the arrangement 
consisting of all hyperplanes with respect to the roots in $\CI^c$.
%the complement of $\CI$ in $\Phi^+$.
Let $\CA(\Phi)$ be the \emph{Weyl arrangement} of $\Phi$,
i.e., $\CA(\Phi) = \{ H_\alpha \mid \alpha \in \Phi^+\}$,
where $H_\alpha$ is the hyperplane in the Euclidean space
$V = \BBR \otimes \BBZ \Phi$ orthogonal to the root $\alpha$.

\begin{defn}[{\cite[\S 11]{sommerstymoczko}}]
\label{def:idealtype}
Let $\CI \subseteq \Phi^+$ be an ideal.
%Let $\CI$ be an ideal in $\Phi^+$.
The \emph{arrangement of ideal type} associated with 
$\CI$ is the subarrangement $\CA_\CI$
of $\CA(\Phi)$ defined by 
\[
\CA_\CI := \{ H_\alpha \mid \alpha \in \CI^c\}.
\]
\end{defn}

%\subsection{Freeness of the arrangements of ideal type}
%\label{ss:freeness}
It was shown by Sommers and Tymoczko \cite[Thm.\ 11.1]{sommerstymoczko}
that in case the root system is classical or of type $G_2$, 
each $\CA_\CI$ is free and
the non-zero exponents 
are given by the ideal exponents of $\CI$,
%see Definition \ref{def:idealexp}.
cf.~\eqref{eq:idealexp}.
The general case was settled
only recently in a uniform manner for all types by 
Abe, Barakat, Cuntz, Hoge, and Terao
in \cite[Thm.\ 1.1]{abeetall:weyl}. 

\begin{theorem}
[{\cite[Thm.\ 11.1]{sommerstymoczko}, \cite[Thm.\ 1.1]{abeetall:weyl}}]
\label{thm:ideals}
Let $\Phi$ be a reduced root system with 
Weyl arrangement $\CA = \CA(\Phi)$.
Then any subarrangement of $\CA$ of
ideal type $\CA_\CI$ is free
with the non-zero
exponents given by the ideal exponents 
$m_i^\CI$ of $\CI$. 
\end{theorem}

The method of proof of 
%Theorem \ref{thm:ideals} in 
\cite[Thm.\ 1.1]{abeetall:weyl}
entails a new general version
of Terao's seminal addition deletion theorem
(Theorem \ref{thm:add-del})
allowing for an entire set of hyperplanes to be added at once 
to a given free arrangement while retaining freeness
- given suitable circumstances - 
as opposed to adding just one hyperplane at a time.
In particular, this method implies that there is a total order on the set of 
hyperplanes in $\CA_\CI$ such that each term of the resulting 
chain of subarrangements of $\CA_\CI$ is itself free.
See \cite{abeterao:freefiltrations} for an application of this 
method in the context of Shi arrangements.

The fact that each $\CA_\CI$ is free with 
the non-zero exponents given by the ideal exponents 
further justifies the choice of terminology in 
Definition \ref{def:idealexp}. 
The proof of \cite[Thm.\ 1.1]{abeetall:weyl} is stunning, for 
not only does it generalize the aforementioned result by 
Shapiro-Steinberg-Kostant, it also gives a new uniform proof of the latter.

\medskip

Thanks to independent fundamental work of 
Arnold and Saito, see \cite[\S 6]{orlikterao:arrangements},  
the reflection arrangement $\CA(W)$ of any 
real reflection group $W$ is free.
In \cite[Cor.\ 5.15]{barakatcuntz:indfree}, 
Barakat and Cuntz 
showed that in fact $\CA(W)$  
is \emph{inductively free}, see Definition \ref{def:indfree}.
The most challenging case here is that of type $E_8$.
In view of these results and considering the method of proof of 
Theorem \ref{thm:ideals} in \cite[Thm.\ 1.1]{abeetall:weyl}, 
it is natural to ask whether the free subarrangements $\CA_\CI$ of 
$\CA(W)$ are also inductively free.
Our first aim is to show that all 
$\CA_\CI$ are indeed inductively free
with the possible exception of some instances only  
in type $E_8$, see Theorem \ref{thm:indfree}.

In recent work \cite{hultman:koszul}, Hultman 
characterized all $\CA_\CI$ that are \emph{supersolvable}, 
see Definition \ref{def:super}. 
Since, supersolvability
implies inductive freeness, see Theorem \ref{thm:superindfree}, 
the following result from \cite{hultman:koszul} readily 
provides a large collection of inductively free $\CA_\CI$.

\begin{theorem}
[{\cite[Thms.~6.6, 7.1]{hultman:koszul}}]
\label{thm:AnBn}
For $\Phi$ of type $A_n$, $B_n$, $C_n$, or $G_2$, 
each $\CA_\CI$ is supersolvable with the non-zero
exponents given by the ideal exponents 
$m_i^\CI$ of $\CI$.
\end{theorem}

\begin{remark}
\label{rem:ss}
For $\Phi$ of type $A_n$, every  
free subarrangement of $\CA$ is 
already supersolvable, \cite[Thm.\ 3.3]{edelmanreiner}.
Thus in that case, Theorem \ref{thm:AnBn} follows 
from Theorem \ref{thm:ideals}.

As opposed to type $A_n$, a
free subarrangement of the supersolvable arrangement of 
type $B_n$ need not be supersolvable in general, e.g.\ 
the Weyl arrangement of type $D_n$ 
is not supersolvable for $n \ge 4$,
cf.\ Theorem \ref{thm:ssW}(i).
So the fact that all $\CA_\CI$ in type $B_n$ 
and type $C_n$ are 
supersolvable is not a consequence of 
Theorem \ref{thm:ideals}.

We give a short proof of Theorem \ref{thm:AnBn} based on 
Theorem \ref{thm:I1I}(i) in Section \S \ref{s:idealtype}. 
\end{remark}

In contrast to the types covered in Theorem \ref{thm:AnBn}, 
there are always non-supersolvable arrangements 
of ideal type for the other Dynkin types,
e.g., for $\CI = \varnothing$, the Weyl arrangement $\CA_\CI = \CA(\Phi)$ 
itself is 
not supersolvable, see Theorem \ref{thm:ssW}; see also \cite{hultman:koszul}.
Nevertheless, in type $D_n$ we obtain the following.

\begin{theorem}
\label{thm:Dn}
Let $\Phi$ be of type $D_n$ for $n \ge 4$.
Then each $\CA_\CI$ is inductively free
with the non-zero
exponents given by the ideal exponents 
$m_i^\CI$ of $\CI$. 
\end{theorem}

From the last two results and Theorem \ref{thm:superindfree}, 
the following is immediate.

\begin{theorem}
\label{thm:classical}
For $\Phi$ of classical type, 
each $\CA_\CI$ is inductively free
with the non-zero
exponents given by the ideal exponents 
$m_i^\CI$ of $\CI$. 
\end{theorem}

For arbitrary $\Phi$, we obtain the following.

\begin{theorem}
\label{thm:penultimate}
For $\theta$ the highest root in $\Phi^+$ and 
$\CI = \{\theta\}$, 
$\CA_\CI$ is inductively free
with the non-zero
exponents given by the ideal exponents 
$m_i^\CI$ of $\CI$.
\end{theorem}

Next we describe an inductive tool that 
is a generalization of a criterion from  
\cite[\S 7]{sommerstymoczko}.
It is crucial in the proofs of both 
Theorems \ref{thm:AnBn} and \ref{thm:Dn}, 
as well as in all subsequent results. 
Using induction, it allows us to show that a large class of 
arrangements of ideal type is inductively free
also in the exceptional types, see Theorems \ref{thm:condition}
and  \ref{thm:indfree-exceptional}.

Let $\Phi_0$ be a (standard) parabolic subsystem of $\Phi$ 
(cf.~\S \ref{ssect:parabolic}) and let 
\[
\Phi^c_0 := \Phi^+ \setminus \Phi_0^+,
\]
the set of  positive roots in the ambient root system which do 
not lie in the smaller one.
The following is a generalization of an inductive criterion from 
\cite[\S 7]{sommerstymoczko}.

\begin{condition} 
\label{cond:linear}
Let $\CI$ be an ideal in $\Phi^+$ and let
$\Phi_0$ be a maximal parabolic subsystem of $\Phi$
such that $\Phi^c_0 \cap \CI^c \ne \varnothing$.
Assume that 
firstly, $\Phi^c_0 \cap \CI^c$ is linearly ordered with respect to 
$\preceq$ so that there is 
a unique root of every occurring 
height in $\Phi^c_0 \cap \CI^c$, 
and secondly, 
for any $\alpha \ne \beta$ 
in $\Phi^c_0 \cap \CI^c$, 
there is a $\gamma \in \Phi_0^+$ so that 
$\alpha, \beta$, and $\gamma$
are linearly dependent.
\end{condition}

\begin{defn}
\label{def:I1}
Fix a standard parabolic subsystem $\Phi_0$ of $\Phi$.
For $\CI$ an ideal in $\Phi^+$, 
\[
\CI_0 := \CI \cap \Phi_0^+
\]
is an ideal in $\Phi_0^+$. Thus 
\[
\CA_{\CI_0} := \{H_\gamma \mid \gamma \in \CI_0^c = \Phi_0^+ \setminus \CI_0\}
\] 
is an arrangement of ideal type in 
$\CA(\Phi_0)$, the Weyl arrangement of $\Phi_0$. 
\end{defn}

Obviously, since 
$\CI_0^c  = \Phi_0^+ \setminus \CI_0 = \CI^c \cap \Phi_0^+ \subseteq \CI^c$, 
we may view 
$\CA_{\CI_0}$ as a
subarrangement 
of $\CA_\CI$
rather than as a subarrangement of $\CA(\Phi_0)$.
Note however, as such, 
$\CA_{\CI_0}$ is not of ideal type in $\CA$ in general,  
since $\CI_0$ need not be an ideal in $\Phi^+$.

Our next result
shows that Condition \ref{cond:linear}
allows us to derive
various stronger freeness properties of $\CA_\CI$
from those of $\CA_{\CI_0}$.
This is the principal inductive tool in 
our entire study.

For the notion of an \emph{inductively factored}
arrangement, see Definition \ref{def:indfactored}.
This also implies inductive freeness, see
Proposition \ref{prop:indfactoredindfree}.

\begin{theorem}
\label{thm:I1I}
Let $\CI$ be an ideal in $\Phi^+$ and let
$\Phi_0$ be a maximal parabolic subsystem of $\Phi$
such that either $\Phi^c_0 \cap \CI^c = \varnothing$
or else Condition \ref{cond:linear} is satisfied.
Then the following hold:
\begin{itemize}
\item[(i)] 
$\CA_{\CI_0}$ is supersolvable
if and only if $\CA_\CI$ is supersolvable;
\item[(ii)] 
$\CA_{\CI_0}$ is inductively free
if and only if $\CA_\CI$ is inductively free;
\item[(iii)] 
$\CA_{\CI_0}$ is inductively factored
if and only if $\CA_\CI$ is inductively factored.
\end{itemize}
In particular, in each of the cases above we have
\[
\exp \CA_\CI = \{\exp \CA_{\CI_0}, |\CA_\CI\setminus \CA_{\CI_0}|\} = 
\{0^{n-k}, m_1^\CI, \ldots, m_s^\CI\}.
\]
\end{theorem} 

The key observation for the proof of Theorem \ref{thm:I1I} 
stems from the fact that 
Condition \ref{cond:linear} entails the presence 
of a modular element in $L(\CA_\CI)$ of rank 
$r(\CA_\CI)-1$ so that 
$\CA_{\CI_0}$ (viewed as a subarrangement of 
$\CA_\CI$) is realized as the localization of 
$\CA_\CI$ with respect to this modular element, 
see Lemmas \ref{lem:ideallocal} and \ref{lem:condition-modular}. 

We emphasize that our proofs of 
Theorems \ref{thm:AnBn} and \ref{thm:Dn} 
crucially depend on Theorem \ref{thm:I1I},  and so do 
our next two results which are central 
concerning inductive freeness of the $\CA_\CI$.

\begin{theorem}
\label{thm:condition}
Let $\Phi$ be a reduced root system 
with Weyl arrangement $\CA = \CA(\Phi)$.
Let $\CI$ be an ideal in $\Phi^+$. 
Suppose that either $\CA_\CI$ is reducible, or else
$\CA_\CI$ is irreducible and there is a 
 maximal parabolic subsystem of $\Phi$ 
such that Condition \ref{cond:linear} is satisfied.
Suppose that each arrangement of ideal type 
for every proper parabolic subsystem is inductively free.
Then the arrangement of ideal type
$\CA_\CI$ is inductively free
with the non-zero
exponents given by the ideal exponents 
$m_i^\CI$ of $\CI$. 
\end{theorem}

Theorem \ref{thm:classical} settles the question 
of inductive freeness of all $\CA_\CI$ for 
classical root systems. Our next results
addresses the situation for the exceptional types.

\begin{theorem}
\label{thm:indfree-exceptional}
Let $\Phi$ be a root system of exceptional type 
with Weyl arrangement $\CA = \CA(\Phi)$.
Let $\CI$ be an ideal in $\Phi^+$. 
Suppose that $\CI$ satisfies one of the following conditions
\begin{itemize}
\item[(i)]
$\CA_\CI$ is reducible;
\item[(ii)]
$\CA_\CI$ is irreducible and there is a 
 maximal parabolic subsystem of $\Phi$ 
such that Condition \ref{cond:linear} is satisfied; or
\item[(iii)]
$\CI = \{\theta\}$ or $\CI = \varnothing$.
\end{itemize}
Suppose that each arrangement of ideal type 
for every proper parabolic subsystem of exceptional type is inductively free.
Then $\CA_\CI$ is inductively free.

Table \ref{table:cond} gives the number of all $\CA_\CI$ 
which satisfy one of the conditions above.
\end{theorem}

Specifically, in Table \ref{table:cond} we present the number of all 
arrangements of ideal type for each exceptional type in the first row,
see \eqref{eq:noI}.
In the second row, we list the number of 
all $\CA_\CI$ which 
satisfy one of the conditions in Theorem \ref{thm:indfree-exceptional}
and thus are inductively free under the assumption 
that this is the case for 
every proper parabolic subsystem of exceptional type.

\begin{table}[ht!b]
\renewcommand{\arraystretch}{1.6}
\begin{tabular}{r|ccccc}
\hline
$\Phi$ & $E_6$ & $E_7$ & $E_8$ & $F_4$ & $G_2$ \\ 
\hline\hline
all $\CA_\CI$ & 833 & 4160 & 25080 & 105 & 8\\
ind. free $\CA_\CI$ & 771 & 3433 & 18902 & 85 & 8\\
\hline
\end{tabular}
\smallskip
\caption{Ind.\ free $\CA_\CI$ for exceptional $\Phi$ from Theorem \ref{thm:indfree-exceptional}}
\label{table:cond}
\end{table}

Thanks to Theorems \ref{thm:classical}
and \ref{thm:indfree-exceptional}, 
it is evident from Table \ref{table:cond} that 
with the possible exception of a relatively small number of
cases in the exceptional types, all $\CA_\CI$ 
are inductively free.
T.~Hoge was able to check 
on a computer that all of the possible exceptions 
in type $F_4$ (20 cases), $E_6$ (62 cases) 
and $E_7$ (727 cases)
are indeed inductively free. 
There are 6178 undecided cases
in $E_8$ at present. 
Therefore, using Theorems \ref{thm:classical}
and \ref{thm:indfree-exceptional} along 
with these computational results, we can derive the following.

\begin{theorem}
\label{thm:indfree}
Let $\Phi$ be a reduced root system with 
Weyl arrangement $\CA = \CA(\Phi)$.
Let $\CI$ be an ideal in $\Phi^+$. 
Then $\CA_\CI$ is inductively free 
with the non-zero
exponents given by the ideal exponents 
$m_i^\CI$ of $\CI$ with the possible 
exception when $W$ is of type $E_8$ 
and $\CI$ is one of $6178$ ideals in $\Phi^+$.
\end{theorem}

This leads to the following conjecture.

\begin{conjecture}
\label{conj:indfree}
Let $\Phi$ be a reduced root system with 
Weyl arrangement $\CA = \CA(\Phi)$.
Then any subarrangement of $\CA$ of
ideal type $\CA_\CI$ is inductively free
with the non-zero
exponents given by the ideal exponents 
$m_i^\CI$ of $\CI$. 
\end{conjecture}

\begin{remarks}

(i).
We emphasize that our proofs of Theorems \ref{thm:AnBn} -- \ref{thm:indfree}
do not depend on Theorem \ref{thm:ideals}.
Instead our arguments are based on extensions of ideas from 
the paper of Sommers and Tymoczko \cite{sommerstymoczko}.  
Condition \ref{cond:linear} is clearly inspired by their work.
Nevertheless, our approach is developed more within the
framework of hyperplane arrangements and 
depends less on the root system combinatorics.

(ii).
The freeness property asserted in 
Conjecture \ref{conj:indfree} is 
the strongest one we can hope 
to hold for \emph{all} $\CA_\CI$ for \emph{all} types,
as the Weyl arrangement $\CA(\Phi)$
itself is supersolvable (inductively factored) 
only if $\Phi$ is of type $A_n$, $B_n$, $C_n$, or $G_2$,  
cf.~Theorem \ref{thm:ssW}.

(iii).
Since inductively free arrangements are 
\emph{divisionally free}, see \cite{abe:divfree}, 
it follows 
from Theorem \ref{thm:indfree} 
that with the possible exception of $6178$
instances in type $E_8$ 
all $\CA_\CI$ are also divisionally free.
Note, 
Conjecture \ref{conj:indfree} would 
settle affirmatively a conjecture by Abe
that all arrangements of ideal type are 
divisionally free, 
\cite[Conj.\ 6.6]{abe:divfree}.

(iv).
The fact that T.~Hoge was able to 
confirm by computer calculations
that the remaining instances 
from Table \ref{table:cond}
in types $F_4$, $E_6$ and $E_7$ are indeed 
inductively free, suggests
that the outstanding $6178$ instances
for $E_8$ might also be in reach 
by suitable computational means.
Although the number of undecided instances 
for $E_8$ is small in relative terms 
(only $6178$ out of $25080$ cases), 
in view of the challenges the authors of 
\cite{barakatcuntz:indfree} were faced with 
in connection with their computational 
proof of the inductive freeness of 
the full Weyl arrangement for $E_8$, 
this is likely going to be a formidable task.

Even if this finite number of unresolved
instances can be confirmed computationally,
it would be very desirable to have
a uniform and conceptional 
proof of Conjecture \ref{conj:indfree}.
This would then entail a conceptual 
proof of the fact that the
Weyl arrangement for $E_8$ itself is 
inductively free. 
\end{remarks}

\subsection{The Poincar\'e polynomial of a Coxeter group}
\label{ssec:coxeter}
Let $W$ be the Coxeter group associated with the root system $\Phi$
and let $\Phi^+$ be the system of positive roots with respect to 
a base $\Pi$. For $w \in W$, we define
\begin{equation}
\label{eq:Nw}
N(w) := \{ \alpha \in \Phi^+ \mid w \alpha \in -\Phi^+\}.
\end{equation}
Then $|N(w)| = \ell(w)$, where $\ell$ is the usual length function 
of $W$ with respect to the fixed set of generators of $W$ corresponding 
to $\Pi$.

Let $t$ be an indeterminate. The \emph{Poincar\'e polynomial} $W(t)$ of 
the Coxeter group $W$ is defined by 
\begin{equation}
\label{eq:poncarecoxeter}
W(t) := \sum_{w \in W} t^{|N(w)|} = \sum_{w \in W} t^{\ell(w)}.
\end{equation}

The following factorization of 
$W(t)$ is due to Solomon \cite{solomon:chevalley}:
\begin{equation}
\label{eq:solomon}
W(t) = \prod_{i=1}^n(1 + t + \ldots + t^{e_i}),
\end{equation}
where $\{e_1, \ldots, e_n\}$ is the 
set of exponents of $W$.
See also Macdonald \cite{macdonald:coxeter}.

In geometric terms, $W(t^2)$ is the Poincar\'e polynomial 
of the flag manifold $G/B$, where $G$ is a semisimple 
Lie group with Weyl group $W$ and $B$ is a Borel subgroup of $G$.
The  formula \eqref{eq:solomon} then gives a well-known 
factorization of this Poincar\'e polynomial.

\subsection{The Poincar\'e polynomial of an ideal}
\label{ss:Weyl}

Fix a subset $R \subset \Phi^+$.
Following \cite[\S 4]{sommerstymoczko}, we say 
that a subset $S \subseteq R$ is \emph{$R$-closed} provided 
if $\alpha, \beta \in S$ and $\alpha + \beta \in R$, 
then also $\alpha + \beta \in S$.
For $\CI$ an ideal in $\Phi^+$,  
we say $S \subset \CI^c$ is of \emph{Weyl type for $\CI$} provided
both $S$ and its complement 
in $\CI^c$ are $\CI^c$-closed.
Let $\CW^\CI$ denote the set of all subsets of $\CI^c$ which are of Weyl type
for $\CI$.
These sets generalize the sets $N(w)$ defined in \eqref{eq:Nw} above. 
For, thanks to \cite[Prop.\ 6.1]{sommerstymoczko},
each such set $S$ of Weyl type for $\CI$ 
is of the form $S = N(w) \cap \CI^c$ for some $w \in W$. 
So in particular, if 
$\CI = \varnothing$, so that $\CI^c = \Phi^+$, the sets of
Weyl type for $\CI = \varnothing$ are 
precisely the sets $N(w)$ in $\Phi^+$; see also \cite{kostant:borelweil}.
Therefore, because of the analogy with \eqref{eq:poncarecoxeter} and following
\cite[\S\S 4, 6]{sommerstymoczko}, we call  
\begin{equation}
\label{eq:st}
\CI(t) := \sum_{S \in \CW^\CI} t^{|S|} 
\end{equation}
the  \emph{Poincar\'e polynomial of $\CI$}.
We can now formulate a further conjecture due to Sommers and Tymoczko
which asserts that the analogue of Solomon's multiplicative 
formula \eqref{eq:solomon} for $W(t)$ holds for $\CI(t)$.

\begin{conjecture}[{\cite[\S 4]{sommerstymoczko}}]
\label{conj:st1}
Let $\CI$ be an ideal in $\Phi^+$. Then 
\begin{equation}
\label{eq:st1}
\CI(t) =  \prod_{i=1}^k(1 + t + \ldots + t^{m^\CI_i}),
\end{equation}
with $m^\CI_i$ the ideal exponents of 
$\CI$ introduced in Definition \ref{def:idealexp}.
\end{conjecture}

For $\CI = \varnothing$ the identity \eqref{eq:st1} specializes to 
Solomon's formula \eqref{eq:solomon}.
In \cite[Thm.\ 4.1]{sommerstymoczko}, 
Sommers and Tymoczko gave a proof of the identity \eqref{eq:st1}
based on case by case arguments in the following cases along with 
computer checks for $F_4$ and $E_6$.

\begin{theorem}
[{\cite[Thm.\ 4.1]{sommerstymoczko}}]
\label{thm:st1-classical}
Let $\Phi$ be of type $A_n$, $B_n$, $C_n$, or $G_2$
and let $\CI$ be an ideal in $\Phi^+$.
Then $\CI(t)$ satisfies \eqref{eq:st1}.  
\end{theorem}

We present an alternative proof of Theorem \ref{thm:st1-classical}
based on 
Theorems \ref{thm:AnBn} and \ref{thm:mult-zeta} and \eqref{eq:ai},
see Remark \ref{rem:weyl-regions}(iii).
Conjecture \ref{conj:st1} is still open for the infinite family of type 
$D_n$ for $n \ge 4$, as well as for types $E_7$ and $E_8$.

In \cite[Thm.\ 9.2]{sommerstymoczko}, 
Sommers and Tymoczko gave a uniform 
proof of \eqref{eq:st1}  for all types 
for $\CI = \{\theta\}$
using an alternate formula for $W(t)$ due to Macdonald, 
\cite[Cor.\ 2.5]{macdonald:coxeter}.

\begin{theorem}
%[{\cite[Thm.\ 9.2]{sommerstymoczko}}]
\label{thm:st1-penultimate}
Let $\theta$ be the highest root of $\Phi^+$ and let 
$\CI = \{\theta\}$. Then 
$\CI(t)$ satisfies \eqref{eq:st1}.  
\end{theorem}

Thanks to work of Tymoczko \cite{tymoczko},
for $\Phi$ of classical type, $\CI(t^2)$
is the Poincar\'e polynomial of the 
\emph{regular nilpotent Hessenberg variety $\CH_\CI$}
associated with $\CI$.
It follows from Theorem \ref{thm:st1-classical} that for 
$\Phi$ of type $A_n$, $B_n$, or $C_n$, 
the Poincar\'e polynomial of $\CH_\CI$ admits a factorization as in 
 \eqref{eq:st1}, see \cite[Thm.\ 10.2]{sommerstymoczko}.

Our second aim is to address Conjecture \ref{conj:st1}. 
In this context  
Condition \ref{cond:linear} is also a crucial inductive tool.
Here are our main results in this direction.

\begin{theorem}
\label{thm:zeta-factors}
Let $\Phi$ be a reduced root system with 
Weyl arrangement $\CA = \CA(\Phi)$.
Let $\CI$ be an ideal in $\Phi^+$.
Suppose that either $\CA_\CI$ is reducible, or else
$\CA_\CI$ is irreducible and there is a 
 maximal parabolic subsystem of $\Phi$
such that Condition \ref{cond:linear} is satisfied.
Suppose that for every proper
parabolic subsystem of $\Phi$ 
the Poincar\'e polynomials of all ideals 
factor as in \eqref{eq:st1}.
Then $\CI(t)$ also factors as in \eqref{eq:st1}, and so  
Conjecture \ref{conj:st1}
holds in these instances.
\end{theorem}

Thanks to Corollary \ref{cor:d4-zeta},
Conjecture \ref{conj:st1} holds for $D_4$. 
So in principle, we can argue by induction for type $D_n$ using  
Theorem \ref{thm:zeta-factors}. Inductively 
Theorem \ref{thm:zeta-factors} covers the bulk of all instances 
in type $D_n$ in the following sense.
It is straightforward 
to see that there are just $2^{n-2}$ ideals $\CI$ in $D_n$  
for which Condition \ref{cond:linear} is not satisfied
with respect to $\Phi_0$ being the standard subsystem of type $D_{n-1}$.
In addition, in three of these $2^{n-2}$ instances 
Conjecture \ref{conj:st1} always holds,
see Remark \ref{rem:st1-dn}.    
This means, assuming the factorization result for $D_{n-1}$, 
it follows from Theorem \ref{thm:zeta-factors}
that it holds for all but $2^{n-2}-3$ instances in $D_n$ as well.
Nevertheless, the general case for type $D_n$ is still unresolved.

The following settles 
Conjecture \ref{conj:st1} for most instances in  
the exceptional types.

\begin{theorem}
\label{thm:zeta-factors-count}
Let $\Phi$ be an irreducible root system of exceptional type
with Weyl arrangement $\CA = \CA(\Phi)$.
Let $\CI$ be an ideal in $\Phi^+$. 
Suppose that $\CI$ satisfies one of the following conditions
\begin{itemize}
\item[(i)]
$\CA_\CI$ is reducible;
\item[(ii)]
$\CA_\CI$ is irreducible and there is a 
 maximal parabolic subsystem of $\Phi$
such that Condition \ref{cond:linear} is satisfied; or 
\item[(iii)]
$\CI = \{\theta\}$ or $\CI = \varnothing$.
\end{itemize}
Suppose that for every proper
parabolic subsystem of $\Phi$ 
the Poincar\'e polynomials of all ideals
factor as in \eqref{eq:st1}.
Then $\CI(t)$ also factors as in \eqref{eq:st1} and so 
Conjecture \ref{conj:st1}
holds in these instances.

Table \ref{table:cond} gives the number of all $\CA_\CI$ 
which satisfy one of the conditions above.
\end{theorem}

For $\Phi$ of exceptional type, 
we see from the data in Table \ref{table:cond}
that there is a large number of
cases when it does follow from 
Theorem \ref{thm:zeta-factors-count} that
$\CI(t)$ factors as in \eqref{eq:st1}.
Sommers and Tymoczko 
have already checked computationally 
that Conjecture \ref{conj:st1} holds 
in types $F_4$ and $E_6$.
It also holds for $D_4$, by Corollary \ref{cor:d4-zeta}.
A.\ Schauenburg was able to confirm this
also for root systems of type
$D_n$ for $5 \le n \le 7$ and $E_7$ by direct computation.
Combining these computational results with 
Theorem \ref{thm:zeta-factors-count}, 
we obtain the following.

\begin{theorem}
\label{thm:zeta-factored-final}
Let $\Phi$ be an irreducible root system of exceptional type
with Weyl arrangement $\CA = \CA(\Phi)$.
Let $\CI$ be an ideal in $\Phi^+$. 
Then Conjecture \ref{conj:st1}
holds for $\CI$ with the possible 
exception when $W$ is of type $E_8$
and $\CI$ is one of $6178$ ideals.
\end{theorem}

\begin{remarks}
\label{rem:zeta-factors}
(i). 
Our proofs of 
Theorems \ref{thm:zeta-factors} and \ref{thm:zeta-factors-count}
utilize the observation from \cite[\S 12]{sommerstymoczko} that 
the Poincar\'e polynomial $\CI(t)$ associated with $\CA_\CI$
coincides with the rank generating function of the poset of regions of 
$\CA_\CI$, see \eqref{eq:ai}.
This allows us to study the former by means of  the latter.

(ii).
Thanks to work of  Bj\"orner, Edelman, and Ziegler 
\cite[Thm.\ 4.4]{bjoerneredelmanziegler}, respectively  
Jambu and Paris \cite[Prop.\ 3.4, Thm.\ 6.1]{jambuparis:factored},
the rank generating function of the poset of regions of 
a real arrangement  which is supersolvable, respectively inductively factored, 
admits a multiplicative decomposition which 
is equivalent to \eqref{eq:st1} for an arrangement of ideal type,
according to \eqref{eq:ai}, 
see Theorem \ref{thm:mult-zeta}. 

Therefore, there is independent interest in 
the supersolvable and inductively factored 
instances among the $\CA_\CI$.
The former have already been characterized 
by Hultman in \cite{hultman:koszul}.

(iii).
In \cite[\S 12]{sommerstymoczko}, 
Sommers and Tymoczko  speculate about 
an equivalence of both of their theorems
\cite[Thms.\ 4.1, 11.1]{sommerstymoczko},
cf.~Theorems \ref{thm:ideals} and \ref{thm:st1-classical}.
The fact that all $\CA_\CI$ are supersolvable in these instances
implies both results, 
cf.~Theorems \ref{thm:AnBn} and \ref{thm:mult-zeta}.
Indeed, 
the similarity of the formulations of  
Theorems \ref{thm:condition}, \ref{thm:indfree-exceptional}, 
\ref{thm:zeta-factors}, and \ref{thm:zeta-factors-count}
confirm a deeper parallelism between both 
conjectures. This connection appears to stem from 
Condition \ref{cond:linear} and the fact that this in turn 
entails the presence 
of a modular element in $L(\CA_\CI)$ of rank 
$r(\CA_\CI)-1$
(cf.~Lemmas \ref{lem:ideallocal} and \ref{lem:condition-modular}), 
which ultimately furnishes the induction argument, 
by means of Theorem \ref{thm:I1I}(ii) and 
Lemma \ref{lem:modular-zeta}. 
\end{remarks}

\subsection{Supersolvable and inductively factored arrangements of ideal type}

It follows from \cite{hultman:koszul} that 
there are always non-supersolvable arrangements of ideal type
in all Dynkin types other than the ones covered in Theorem \ref{thm:AnBn}.
Likewise for the notion of inductive factoredness.
In view of Conjecture \ref{conj:st1} it
is rather natural to investigate inductive factoredness
for the $\CA_\CI$, cf.~Remark \ref{rem:zeta-factors}(ii).
Theorem \ref{thm:I1I}(iii) proves to be
an equally useful inductive tools in this regard.
We give an indication 
for $\Phi$ of type $D_n$ for $n \ge 4$ and 
$E_6$, $E_7$, and $E_8$, where the situation is already  
quite different to the one covered in Theorem \ref{thm:AnBn}, 
as demonstrated in our next result.
 
Let $\theta$ be the highest root in $\Phi$.
Then $h = \hgt(\theta) + 1$ is the 
\emph{Coxeter number} of $W$. For $1 \le t \le h$, let 
$\CI_t$ be the ideal consisting of all roots of height at least $t$, i.e.
\[
\CI_t :=\{\alpha \in \Phi^+ \mid \hgt(\alpha) \ge t\}
\]
In particular, we have $\CI_1 = \Phi^+$ and $\CI_h = \varnothing$.
%Note that the empty ideal $\CI_h = \varnothing$ 
%is contained in every $\CI_t$ for $1 \le t < h$. 

\begin{theorem}
\label{thm:d4notss}
Let $\Phi$ be of type $D_n$ for $n \ge 4$, $E_6$, $E_7$, or $E_8$.
Let $\CI$ be an ideal in $\Phi^+$.
Then the following hold:
\begin{itemize}
\item[(i)]
if $\CI \subseteq \CI_4$, 
then $\CA_\CI$ is not supersolvable;
\item[(ii)]
if $\CI \subseteq \CI_5$, 
then $\CA_\CI$ is not inductively factored.
\end{itemize}
\end{theorem}

\begin{theorem}
\label{thm:d4ss}
Let $\Phi$ be of type $D_n$ for $n \ge 4$, $E_6$, $E_7$, or $E_8$.
Let $\CI$ be an ideal in $\Phi^+$.
Then the following hold:
\begin{itemize}
\item[(i)]
if $\CI \supseteq \CI_3$, 
then $\CA_\CI$ is supersolvable;
\item[(ii)]
if $\CI \supseteq \CI_4$, 
then $\CA_\CI$ is inductively factored.
\end{itemize}
\end{theorem}

Thanks to 
Theorem \ref{thm:superindfree} and Proposition \ref{prop:indfactoredindfree}, 
in all instances covered in Theorem \ref{thm:d4ss}, 
$\CA_\CI$ is inductively free.

While in types $A_n$, $B_n$, $C_n$ and $G_2$, 
the notions of supersolvability and 
inductive factoredness coincide for all $\CA_\CI$, 
according to Theorem \ref{thm:AnBn}, 
in contrast,  
in type $D_n$ for $n \ge 4$, $E_6$, $E_7$, or $E_8$, the 
arrangement 
$\CA_{\CI_4}$ is inductively factored but not supersolvable,
thanks to Theorems \ref{thm:d4notss}(i) and \ref{thm:d4ss}(ii).
Note that Theorems \ref{thm:d4notss}(i) and \ref{thm:d4ss}(i)
also both follow readily from \cite{hultman:koszul}.

\bigskip

The paper is organized as follows. In \S\S \ref{ssect:arrangements}
-- \ref{ssect:modular} we recall some basic terminology and introduce 
further notation 
on hyperplane arrangements and record some basic facts on modular elements
in the lattice of intersections of an arrangement.
This is followed by brief sections on the fundamental notion on free
and inductively free arrangements, including Terao's addition deletion Theorem 
\ref{thm:add-del}.
In \S\S \ref{ssect:supersolve} and \ref{ssect:factored}, 
the concepts of supersolvable, nice and inductively factored
arrangements are recalled.

It is worth noting that 
all of these properties above are inherited by arbitrary localizations. 
In Lemmas \ref{lem:modular-indfree} and \ref{lem:modular-indfactored}
and Corollary \ref{cor:modular-super}, 
we show that if $X$ in $L(\CA)$ is modular of rank 
$r(\CA)-1$ and the localization $\CA_X$ satisfies any of these properties,
the so does $\CA$ itself.

The number of all 
ideals $\CI$ in $\Phi^+$ was  
first obtained in a case by case analysis by Shi \cite{shi:signtype}. 
An expression for this number 
in closed form was proved by 
Cellini and Papi 
\cite{cellinipapi:nilpotentI, cellinipapi:nilpotentII}, 
see  \S \ref{sect:noideals}.
We also record closed formulas for the former sequences for 
the classical types depending on the rank in 
Table \ref{table:no-ideals-classical}.
They coincide with the famous Catalan sequences.
Concerning the mathematical ubiquity of the latter, see \cite{stanley:book2}.

In the main section 
of the paper, \S \ref{s:idealtype}, we prove 
in Lemma \ref{lem:ideallocal} 
that $\CA_{\CI_0}$ is always a localization of $\CA_\CI$
and in Lemma \ref{lem:condition-modular} that 
if $\CI \subseteq \Phi^+$ and 
$\Phi_0$ satisfy Condition \ref{cond:linear},
then the center
of this localization is a modular element of rank
$r(\CA_\CI)-1$.
This and further results stemming from this fact
are then used to prove Theorem \ref{thm:I1I}.
Subsequently Theorems \ref{thm:condition} and \ref{thm:AnBn}
are then derived as consequences of 
Theorem \ref{thm:I1I}.
This is followed by a proof of Theorem \ref{thm:Dn}.
All of the above crucially depend on 
Condition \ref{cond:linear} and 
Theorem \ref{thm:I1I}.

Section \ref{S:indfree} is devoted to the proof of 
Theorem \ref{thm:indfree-exceptional}.
This crucially depends again 
on Theorems \ref{thm:AnBn} and \ref{thm:I1I}.

In \S \ref{s:rankgenerating}, we recall the basics on the rank-generating
function $\zeta(P(\CA,B), t)$ of the poset of regions 
$P(\CA, B)$ of a real arrangement $\CA$.
In Theorem \ref{thm:mult-zeta} we recall 
theorems of Bj\"orner, Edelman and Ziegler
\cite[Thm.\ 4.4]{bjoerneredelmanziegler}, respectively
Jambu and Paris
\cite[Prop.\ 3.4, Thm.\ 6.1]{jambuparis:factored},
asserting that for $\CA$ supersolvable, respectively 
inductively factored,  
$\zeta(P(\CA,B), t)$ satisfies a factorization
analogous to \eqref{eq:st1}.
The proof of \cite[Thm.\ 4.4]{bjoerneredelmanziegler} 
is used in Lemma \ref{lem:modular-zeta}
to show that 
$\zeta(P(\CA,B), t)$ factors with respect to a 
localization of $\CA$ at a modular element of rank 
$r(\CA) - 1$.
This in turn allows us to derive the desired factorization
of $\zeta(P(\CA,B), t)$ despite the absence of 
supersolvability or inductive factoredness of the ambient arrangement, 
cf.~Example \ref{ex:modular-zeta}.
Lemma \ref{lem:modular-zeta} quickly furnishes the proofs of 
Theorems \ref{thm:zeta-factors}
and \ref{thm:zeta-factors-count}.

Finally, in \S \ref{S:ssAI}, 
after classifying all supersolvable and all inductively factored 
$\CA_\CI$ in type $D_4$ in Lemma \ref{lem:d4}, 
we complete the proofs of 
Theorems \ref{thm:d4notss} and \ref{thm:d4ss}.
The proof of Theorem \ref{thm:d4ss} utilizes 
Condition \ref{cond:linear} and 
Theorem \ref{thm:I1I} once again.

For general information about arrangements, Weyl groups and root systems,  
we refer the reader to \cite{bourbaki:groupes} and 
\cite{orlikterao:arrangements}.

\section{Recollections and Preliminaries}
\label{sect:prelims}

\subsection{Hyperplane arrangements}
\label{ssect:arrangements}
Let $\BBK$ be a field and let
$V = \BBK^\ell$ be an $\ell$-dimensional $\BBK$-vector space.
A \emph{hyperplane arrangement} $\CA = (\CA, V)$ in $V$ 
is a finite collection of hyperplanes in $V$ each 
containing the origin of $V$.
We also use the term $\ell$-arrangement for $\CA$. 
We denote the empty arrangement in $V$ by $\varnothing_\ell$.

The \emph{lattice} $L(\CA)$ of $\CA$ is the set of subspaces of $V$ of
the form $H_1\cap \dotsm \cap H_i$ where $\{ H_1, \ldots, H_i\}$ is a subset
of $\CA$. 
For $X \in L(\CA)$, we have two associated arrangements, 
firstly
$\CA_X :=\{H \in \CA \mid X \subseteq H\} \subseteq \CA$,
the \emph{localization of $\CA$ at $X$}, 
and secondly, 
the \emph{restriction of $\CA$ to $X$}, $(\CA^X,X)$, where 
$\CA^X := \{ X \cap H \mid H \in \CA \setminus \CA_X\}$.
Note that $V$ belongs to $L(\CA)$
as the intersection of the empty 
collection of hyperplanes and $\CA^V = \CA$. 
The lattice $L(\CA)$ is a partially ordered set by reverse inclusion:
$X \le Y$ provided $Y \subseteq X$ for $X,Y \in L(\CA)$.

More generally, for $U$ an arbitrary subspace of $V$, define  
$\CA_U :=\{H \in \CA \mid U \subseteq H\} \subseteq \CA$, the 
\emph{localization of $\CA$ at $U$}.
Note that for $X = \cap_{H \in \CA_U} H$, we have 
$\CA_X = \CA_U$ and $X$ belongs to the 
intersection lattice of $\CA$. 

For $\CA \ne \varnothing_\ell$, 
let $H_0 \in \CA$.
Define $\CA' := \CA \setminus\{ H_0\}$,
and $\CA'' := \CA^{H_0} = \{ H_0 \cap H \mid H \in \CA'\}$,
the restriction of $\CA$ to $H_0$.
Then $(\CA, \CA', \CA'')$ is a \emph{triple} of arrangements,
\cite[Def.\ 1.14]{orlikterao:arrangements}. 

Throughout, we only consider arrangements $\CA$
such that $0 \in H$ for each $H$ in $\CA$.
These are called \emph{central}.
In that case the \emph{center} 
$T(\CA) := \cap_{H \in \CA} H$ of $\CA$ is the unique
maximal element in $L(\CA)$  with respect
to the partial order.
A \emph{rank} function on $L(\CA)$
is given by $r(X) := \codim_V(X)$.
The \emph{rank} of $\CA$ 
is defined as $r(\CA) := r(T(\CA))$.

The \emph{product}
$\CA = (\CA_1 \times \CA_2, V_1 \oplus V_2)$ 
of two arrangements $(\CA_1, V_1), (\CA_2, V_2)$
is defined by
\begin{equation*}
\label{eq:product}
\CA = \CA_1 \times \CA_2 := \{H_1 \oplus V_2 \mid H_1 \in \CA_1\} \cup 
\{V_1 \oplus H_2 \mid H_2 \in \CA_2\},
\end{equation*}
see \cite[Def.\ 2.13]{orlikterao:arrangements}.
%In particular,  $|\CA| = |\CA_1| + |\CA_2|$. 

\subsection{Modular elements in $L(\CA)$}
\label{ssect:modular}

%Let $\CA$ be an arrangement.
We say that $X \in L(\CA)$ is \emph{modular}
provided $X + Y \in L(\CA)$ for every $Y \in L(\CA)$,
 \cite[Cor.\ 2.26]{orlikterao:arrangements}.
We require the following characterization of modular members of $L(\CA)$ of rank $r-1$,
see the proof of \cite[Thm.\ 4.3]{bjoerneredelmanziegler}.

\begin{lemma}
\label{lem:modular1}
Let $\CA$ be an arrangement of rank $r$.
Suppose that $X \in L(\CA)$ is of rank $r-1$.
Then $X$ is modular if and only if 
for any two distinct $H_1, H_2 \in \CA \setminus \CA_X$ there
is a $H_3 \in \CA_X$ so that $r(H_1 \cap H_2 \cap H_3) = 2$,
i.e.\ $H_1, H_2, H_3$ are linearly dependent.
\end{lemma}

Next we record a special case of a general fact about modular elements in a 
geometric lattice, cf.~\cite[Prop.\ 2.42]{aigner:combinatorialtheory}
or \cite[Lem.\ 2.27]{orlikterao:arrangements}.

\begin{lemma}
\label{lem:modular2}
Let $\CA$ be an arrangement of rank $r$.
Suppose that $X \in L(\CA)$ is modular of rank $r-1$.
Then the map $L(\CA_X) \to L(\CA^H)$ given by $Y \mapsto Y \cap H$
is a lattice isomorphism for any $H \in \CA \setminus \CA_X$.
In particular, $\CA_X \cong \CA^H$.
\end{lemma}

\subsection{Free hyperplane arrangements}
\label{ssect:free}
Let $S = S(V^*)$ be the symmetric algebra of the dual space $V^*$ of $V$.
Let $\Der(S)$ be the $S$-module of $\BBK$-derivations of $S$.
Since $S$ is graded, 
$\Der(S)$ is a graded $S$-module.

Let $\CA$ be an arrangement in $V$. 
Then for $H \in \CA$ we fix $\alpha_H \in V^*$ with
$H = \ker \alpha_H$.
The \emph{defining polynomial} $Q(\CA)$ of $\CA$ is given by 
$Q(\CA) := \prod_{H \in \CA} \alpha_H \in S$.
The \emph{module of $\CA$-derivations} of $\CA$ is 
defined by 
\[
D(\CA) := \{\theta \in \Der(S) \mid \theta(Q(\CA)) \in Q(\CA) S\} .
\]
We say that $\CA$ is \emph{free} if 
$D(\CA)$ is a free $S$-module, cf.\ \cite[\S 4]{orlikterao:arrangements}.

If $\CA$ is a free arrangement, then the $S$-module
$D(\CA)$ admits a basis of $\ell$ homogeneous derivations, 
say $\theta_1, \ldots, \theta_\ell$, \cite[Prop.\ 4.18]{orlikterao:arrangements}.
While the $\theta_i$'s are not unique, their polynomial 
degrees $\pdeg \theta_i$ 
are unique (up to ordering). This multiset is the set of 
\emph{exponents} of the free arrangement $\CA$
and is denoted by $\exp \CA$.

Terao's celebrated \emph{addition deletion theorem} 
which we recall next
plays a 
pivotal role in the study of free arrangements, 
\cite[\S 4]{orlikterao:arrangements}.

\begin{theorem}[\cite{terao:freeI}]
\label{thm:add-del}
Suppose that $\CA$ is non-empty.
Let  $(\CA, \CA', \CA'')$ be a triple of arrangements. Then any 
two of the following statements imply the third:
\begin{itemize}
\item[(i)] $\CA$ is free with $\exp \CA = \{ b_1, \ldots , b_{\ell -1}, b_\ell\}$;
\item[(ii)] $\CA'$ is free with $\exp \CA' = \{ b_1, \ldots , b_{\ell -1}, b_\ell-1\}$;
\item[(iii)] $\CA''$ is free with $\exp \CA'' = \{ b_1, \ldots , b_{\ell -1}\}$.
\end{itemize}
\end{theorem}
There are various stronger notions of freeness
which we discuss in the following subsections.

\subsection{Inductively free arrangements}
\label{ssect:indfree}

Theorem \ref{thm:add-del}
motivates the notion of 
\emph{inductively free} arrangements,  see 
\cite{terao:freeI} or 
\cite[Def.\ 4.53]{orlikterao:arrangements}.

\begin{defn}
\label{def:indfree}
The class $\CIF$ of \emph{inductively free} arrangements 
is the smallest class of arrangements subject to
\begin{itemize}
\item[(i)] $\varnothing_\ell$ belongs to $\CIF$, for every $\ell \ge 0$;
\item[(ii)] if there exists a hyperplane $H_0 \in \CA$ such that both
$\CA'$ and $\CA''$ belong to $\CIF$, and $\exp \CA '' \subseteq \exp \CA'$, 
then $\CA$ also belongs to $\CIF$.
\end{itemize}
\end{defn}

\begin{remark}
\label{rem:indfreetable}
It is possible to describe an inductively free arrangement $\CA$ by means of 
a so called 
\emph{induction table}, cf.~\cite[\S 4.3, p.~119]{orlikterao:arrangements}.
In this process we start with an inductively free arrangement
and add hyperplanes successively ensuring that 
part (ii) of Definition \ref{def:indfree} is satisfied.
This process is referred to as \emph{induction of hyperplanes}.
This procedure amounts to 
choosing a total order on $\CA$, say 
$\CA = \{H_1, \ldots, H_m\}$, 
so that each of the subarrangements 
$\CA_i := \{H_1, \ldots, H_i\}$
and each of the restrictions $\CA_i^{H_i}$ is inductively free
for $i = 1, \ldots, m$.
In the associated induction table we record in the $i$-th row the information 
of the $i$-th step of this process, by 
listing $\exp \CA_i' = \exp \CA_{i-1}$, 
$H_i$, 
as well as $\exp \CA_i'' = \exp \CA_i^{H_i}$, 
for $i = 1, \ldots, m$.
For instance, see 
\cite[Tables 4.1, 4.2]{orlikterao:arrangements}, or 
Table \ref{indtable:modular1} below. 
\end{remark}

The class of free arrangements is closed with respect to taking 
localizations, 
cf.~\cite[Thm.\ 4.37]{orlikterao:arrangements}. 
This also holds for the class $\CIF$, 
\cite[Thm.\ 1.1]{hogeroehrleschauenburg:free}.

\begin{proposition}
%[{\cite[Thm.\ 1.1]{hogeroehrleschauenburg:free}}]
\label{prop:indfreelocal}
If $\CA$ is inductively free,
then so is $\CA_U$ for every subspace $U$ in $V$.
\end{proposition}

For  $X$ in $L(\CA)$ modular of rank $r-1$, 
we get the following converse to Proposition \ref{prop:indfreelocal},
see also \cite[Lem.\ 6.7]{slofstra:schubert}.

\begin{lemma}
\label{lem:modular-indfree}
Let $\CA$ be an arrangement of rank $r$.
Suppose that $X \in L(\CA)$ is modular of rank $r-1$.
If $\CA_X$ is inductively free, then so is $\CA$.
In particular, if $\exp \CA_X = \{0, e_1, \ldots, e_{\ell-1}\}$,
then $\exp \CA = \{e_1, \ldots, e_{\ell-1}, e_\ell\}$,
where $e_\ell := | \CA \setminus \CA_X|$.
\end{lemma}

\begin{proof}
We argue by means of 
an induction table for $\CA$ starting with the inductively free 
subarrangement $\CA_X$ and adding hyperplanes 
from $\CA \setminus \CA_X = \{H_1, \ldots, H_{e_\ell}\}$ 
(in any fixed order) successively, 
see Remark \ref{rem:indfreetable}.
Let $\CA_i := \CA_X \cup \{H_1, \ldots, H_i\}$ for $i = 1, \ldots, e_\ell$.
Then, since $X$ is modular of rank $r-1$ also in $L(\CA_i)$, 
for $i = 1, \ldots, e_\ell$ (cf.~the second part of the 
proof of \cite[Thm.\ 4.3]{bjoerneredelmanziegler}),
it follows from 
Lemma \ref{lem:modular2} that $\CA_i^{H_i}$ is 
isomorphic to $\CA_X$ for each $i = 1, \ldots, e_\ell$.

\begin{table}[h]
\renewcommand{\arraystretch}{1.3}
\begin{tabular}{lll}  \hline
  $\exp \CA_i'$ & $H_i$ & $\exp \CA_i''$\\
\hline\hline
$e_1, \ldots, e_{\ell-1}, 0$ 	   &    ${H_1}$       {\hglue 5pt}  & $e_1, \ldots, e_{\ell-1}$ \\
$e_1, \ldots, e_{\ell-1}, 1$	   &    ${H_2}$        & $e_1, \ldots, e_{\ell-1}$ \\
$\vdots$                                  &   $\vdots$                     &          $\vdots$     \\
$e_1, \ldots, e_{\ell-1}, e_\ell-1$	  {\hglue 5pt}  &    ${H_{e_\ell}}$    & $e_1, \ldots, e_{\ell-1}$ \\
$e_1, \ldots, e_{\ell-1}, e_\ell$ 	   &                                     &         \\
\hline                                
\end{tabular}\\
\smallskip
\caption{Induction table for $\CA$ starting at $\CA_X$} 
\label{indtable:modular1} 
\end{table}

It thus follows from Table 
\ref{indtable:modular1} and 
a repeated application of the addition part of 
Theorem \ref{thm:add-del} that $\CA$ is inductively free with 
$\exp \CA = \{e_1, \ldots, e_{\ell-1}, e_\ell\}$.
\end{proof}

\begin{remark}
\label{rem:modular-free}
The same argument as
the one in the proof of 
Lemma \ref{lem:modular-indfree} 
shows that if 
$X \in L(\CA)$ is modular of rank $r-1$
and $\CA_X$ is free, then so is $\CA$.
\end{remark}

Free arrangements behave well with respect to 
the  product construction for arrangements, 
\cite[Prop.\ 4.28]{orlikterao:arrangements}.
This property descends to 
the class $\CIF$,
\cite[Prop.\ 2.10]{hogeroehrle:indfree}. 

\begin{proposition}
\label{prop:product-indfree}
Let $\CA_1, \CA_2$ be two arrangements.
Then  $\CA = \CA_1 \times \CA_2$ is inductively free
if and only if both 
$\CA_1$ and $\CA_2$ are inductively free.
\end{proposition}

\subsection{Supersolvable arrangements}
\label{ssect:supersolve}
The following notion is due to Stanley \cite{stanley:super}. 

\begin{defn}
%[{\cite{stanley:super}}]
\label{def:super}
Let $\CA$ be a central arrangement of rank $r$.
We say that $\CA$ is 
\emph{supersolvable} 
provided there is a maximal chain
\[
V = X_0 < X_0 < \ldots < X_{r-1} < X_r = T(\CA)
\]
 of modular elements $X_i$ in $L(\CA)$,
cf.\ \cite[Def.\ 2.32]{orlikterao:arrangements}.
\end{defn}

The class of supersolvable arrangements is closed under 
localization, {\cite{stanley:super}.

\begin{proposition}
%[{\cite{stanley:super}}]
\label{prop:superlocal}
If $\CA$ is supersolvable, 
then so is $\CA_U$ for every subspace $U$ in $V$.
\end{proposition}

We require the following characterization 
of supersolvable arrangements due to 
Bj\"orner, Edelman and Ziegler, 
\cite[Thm.\ 4.3]{bjoerneredelmanziegler}.

\begin{theorem}
%[{\cite[Thm.\ 4.3]{bjoerneredelmanziegler}}]
\label{thm:super}
Every arrangement of rank at most $2$ is supersolvable.
Let $\CA$ be an arrangement of rank $r \ge 3$.
Then  $\CA$ is supersolvable
if and only if $\CA$ can be
written as the proper disjoint union of two subarrangements
$\CA = \CA_0 \coprod \CA_1$, such that $\CA_0$ is supersolvable
of rank $r-1$ and for any $H_1 \ne H_2$ in $\CA_1$,
there is a $H_3$ in  $\CA_0$ so that
$H_1 \cap H_2 \subseteq H_3$.
\end{theorem}

Thanks to Lemma \ref{lem:modular1} and Theorem \ref{thm:super},
we get a converse to Proposition \ref{prop:superlocal}
for $X$ in $L(\CA)$ modular of rank $r-1$.
(This is just the reverse implication in 
Theorem \ref{thm:super}.)

\begin{corollary}
\label{cor:modular-super}
Let $\CA$ be an arrangement of rank $r$.
Suppose that $X \in L(\CA)$ is modular of rank $r-1$.
If $\CA_X$ is supersolvable, then so is $\CA$.
In particular, if $\exp \CA_X = \{0, e_1, \ldots, e_{\ell-1}\}$,
then $\exp \CA = \{e_1, \ldots, e_{\ell-1}, e_\ell\}$,
where $e_\ell := | \CA \setminus \CA_X|$.
\end{corollary}

Also supersolvable arrangements are 
compatible with the  
product construction for arrangements.

\begin{proposition}
[{\cite[Prop.\ 2.5]{hogeroehrle:super}}]
\label{prop:product-super}
Let $\CA_1, \CA_2$ be two arrangements.
Then  $\CA = \CA_1 \times \CA_2$ is supersolvable
if and only if both 
$\CA_1$ and $\CA_2$ are supersolvable.
\end{proposition}

The connection of this notion with freeness is
due to Jambu and Terao.

\begin{theorem}
[{\cite[Thm.\ 4.2]{jambuterao:free}}]
\label{thm:superindfree}
A supersolvable arrangement is inductively free.
\end{theorem}

\subsection{Nice and inductively factored arrangements}
\label{ssect:factored}

The notion of a \emph{nice} or \emph{factored} 
arrangement goes back to Terao \cite{terao:factored}.
It generalizes the concept of a supersolvable arrangement, see
\cite[Thm.\ 5.3]{orliksolomonterao:hyperplanes} and 
\cite[Prop.\ 2.67, Thm.\ 3.81]{orlikterao:arrangements}.
Terao's main motivation was to give a 
general combinatorial framework to 
deduce tensor factorizations of the underlying Orlik-Solomon algebra,
see also \cite[\S 3.3]{orlikterao:arrangements}.
We recall the relevant notions  
from \cite{terao:factored}
(cf.\  \cite[\S 2.3]{orlikterao:arrangements}):

\begin{defn}
\label{def:factored}
Let $\pi = (\pi_1, \ldots , \pi_s)$ be a partition of $\CA$.
\begin{itemize}
\item[(a)]
$\pi$ is called \emph{independent}, provided 
for any choice $H_i \in \pi_i$ for $1 \le i \le s$,
the resulting $s$ hyperplanes are linearly independent, i.e.\
$r(H_1 \cap \ldots \cap H_s) = s$.
\item[(b)]
Let $X \in L(\CA)$.
The \emph{induced partition} $\pi_X$ of $\CA_X$ is given by the non-empty 
blocks of the form $\pi_i \cap \CA_X$.
\item[(c)]
$\pi$ is
\emph{nice} for $\CA$ or a \emph{factorization} of $\CA$  provided 
\begin{itemize}
\item[(i)] $\pi$ is independent, and 
\item[(ii)] for each $X \in L(\CA) \setminus \{V\}$, the induced partition $\pi_X$ admits a block 
which is a singleton. 
\end{itemize}
\end{itemize}
If $\CA$ admits a factorization, then we also say that $\CA$ is \emph{factored} or \emph{nice}.
\end{defn}

\begin{remark}
\label{rem:factored}
If $\CA$ is non-empty and   
$\pi$ is a nice partition of $\CA$, then the non-empty parts of the 
induced partition $\pi_X$ form a nice partition of $\CA_X$
for each $X \in L(\CA)\setminus\{V\}$;
cf.~the proof of \cite[Cor.\ 2.11]{terao:factored}.
\end{remark}

Following Jambu and Paris 
\cite{jambuparis:factored}, 
we introduce further notation.
Suppose $\CA$ is not empty. 
Let $\pi = (\pi_1, \ldots, \pi_s)$ be a partition of $\CA$.
Let $H_0 \in \pi_1$ and 
let $(\CA, \CA', \CA'')$ be the triple associated with $H_0$. 
Then $\pi$ induces a 
partition 
$\pi'$ of 
$\CA'$, i.e.\ the non-empty 
subsets $\pi_i \cap \CA'$.
Note that since $H_0 \in \pi_1$, we have
$\pi_i \cap \CA' = \pi_i$ 
for $i = 2, \ldots, s$. 
Also, associated with $\pi$ and $H_0$, we define 
the \emph{restriction map}
\[
\R := \R_{\pi,H_0} : \CA \setminus \pi_1 \to \CA''\ \text{ given by } \ H \mapsto H \cap H_0
\]
and set 
\[
\pi_i'' := \R(\pi_i) = \{H \cap H_0 \mid H \in \pi_i\} \
\text{ for }\  2 \le i \le s.
\]
In general, $\R$ need not be surjective nor injective.
However, since we are only concerned with cases when 
$\pi'' = (\pi_2'', \ldots, \pi_s'')$ is a
partition of $\CA''$,  
$\R$ has to be onto and 
$\R(\pi_i) \cap \R(\pi_j) = \varnothing$ for $i \ne j$.

The following gives an  
analogue of Terao's 
addition deletion Theorem \ref{thm:add-del} for 
free arrangements for the class of 
nice arrangements.

\begin{theorem}
[{\cite[Thm.\ 3.5]{hogeroehrle:factored}}]
\label{thm:add-del-factored}
Suppose that $\CA \ne \varnothing_\ell$.
Let $\pi = (\pi_1, \ldots, \pi_s)$ be a  partition  of $\CA$.
Let $H_0 \in \pi_1$ and 
let $(\CA, \CA', \CA'')$ be the triple associated with $H_0$. 
Then any two of the following statements imply the third:
\begin{itemize}
\item[(i)] $\pi$ is nice for $\CA$;
\item[(ii)] $\pi'$ is nice for $\CA'$;
\item[(iii)] $\R: \CA \setminus \pi_1 \to \CA''$ 
is bijective and $\pi''$ is nice for $\CA''$.
\end{itemize}
\end{theorem}

The bijectivity condition on $\R$ 
in the theorem is necessary, 
cf.~\cite[Ex.\ 3.3]{hogeroehrle:factored}.
Theorem \ref{thm:add-del-factored} 
motivates
the following stronger notion of factorization, 
cf.\ \cite{jambuparis:factored}, \cite[Def.\ 3.8]{hogeroehrle:factored}.

\begin{defn} 
\label{def:indfactored}
The class $\CIFAC$ of \emph{inductively factored} arrangements 
is the smallest class of pairs $(\CA, \pi)$ of 
arrangements $\CA$ together with a partition $\pi$
subject to
\begin{itemize}
\item[(i)] $(\varnothing_\ell, ( ))$ belongs to $\CIFAC$ for each $\ell \ge 0$;
\item[(ii)] if there exists a partition $\pi$ of $\CA$ 
and a hyperplane $H_0 \in \pi_1$ such that 
for the triple $(\CA, \CA', \CA'')$ associated with $H_0$ 
the restriction map $\R = \R_{\pi, H_0} : \CA \setminus \pi_1 \to \CA''$ 
is bijective and for the induced partitions $\pi'$ of $\CA'$ and $\pi''$ of $\CA''$ 
both $(\CA', \pi')$ and $(\CA'', \pi'')$ belong to $\CIFAC$, 
then $(\CA, \pi)$ also belongs to $\CIFAC$.
\end{itemize}
If $(\CA, \pi)$ is in $\CIFAC$, then we say that
$\CA$ is \emph{inductively factored with respect to $\pi$}, or else
that $\pi$ is an \emph{inductive factorization} of $\CA$. 
Sometimes, we simply say $\CA$ is \emph{inductively factored} without 
reference to a specific inductive factorization of $\CA$.
\end{defn}

The connection with the previous notions is as follows, \cite[Prop.\ 3.11]{hogeroehrle:factored}.

\begin{proposition}
%[{\cite[Prop.\ 3.11]{hogeroehrle:factored}}]
\label{prop:superindfactored}
If $\CA$ is supersolvable, then $\CA$ is inductively factored.
\end{proposition}

\begin{proposition}
[{\cite[Prop.\ 2.2]{jambuparis:factored}, 
\cite[Prop.\ 3.14]{hogeroehrle:factored}}]
\label{prop:indfactoredindfree}
Let $\pi = (\pi_1, \ldots, \pi_r)$ be an inductive factorization of $\CA$. 
Then $\CA$ is inductively free with
$\exp \CA = \{0^{\ell-r}, |\pi_1|, \ldots, |\pi_r|\}$.
\end{proposition}

\begin{remark}
\label{rem:indfactable}
In analogy to $\CIF$, for members in $\CIFAC$ 
one can present a so called induction table of factorizations, 
cf.~\cite[Rem.\ 3.16]{hogeroehrle:factored}.

The proof of Proposition \ref{prop:indfactoredindfree} 
shows 
that if $\pi$ is an inductive factorization of $\CA$
and $H_0 \in \CA$ is distinguished with respect to $\pi$, then 
the triple $(\CA, \CA', \CA'')$ with respect to $H_0$ 
is a triple of inductively free arrangements.
Thus an induction table of $\CA$ can be constructed,
compatible with suitable inductive factorizations of 
the subarrangements $\CA_i$.

Let $\CA = \{H_1, \ldots, H_m\}$ be a 
choice of a total order on $\CA$.
Then, starting with the empty partition for 
$\varnothing_\ell$, we can attempt to
build inductive factorizations $\pi_i$ of $\CA_i$
consecutively, resulting in an inductive factorization 
$\pi = \pi_m$ of $\CA = \CA_m$.
This is achieved by invoking Theorem \ref{thm:add-del-factored}
repeatedly in order to derive that each $\pi_i$ is an inductive factorization of $\CA_i$.

We then add the inductive factorizations $\pi_i$
of $\CA_i$ as additional data into an induction table for $\CA$.
The data in such an extended induction table 
together with the addition part of Theorem \ref{thm:add-del-factored}
then proves that $\CA$ is inductively factored. 
We refer to this technique as \emph{induction of factorizations} and the
corresponding table as an \emph{induction table of factorizations} for $\CA$.
See for instance \cite[\S 3.3]{hogeroehrle:factored}, or Table  \ref{indtable:modular2}.
\end{remark}

By Remark \ref{rem:factored}, the class of nice arrangements
is closed with respect to taking localizations.
This property restricts to the class $\CIFAC$, \cite[Thm.\ 1.1]{muellerroehrle:factored}.

\begin{proposition}
%[{\cite[Thm.\ 1.1]{muellerroehrle:factored}}]
\label{prop:indfaclocal}
If $\CA$ is inductively factored,
then so is $\CA_U$ for every subspace $U$ in $V$.
\end{proposition}

The following gives a converse to Proposition \ref{prop:indfaclocal}
for $X$ in $L(\CA)$  modular of rank $r-1$.

\begin{lemma}
\label{lem:modular-indfactored}
Let $\CA$ be an arrangement of rank $r$.
Suppose that $X \in L(\CA)$ is modular of rank $r-1$.
If $\CA_X$ is inductively factored, then so is $\CA$.
In particular, if $\exp \CA_X = \{0, e_1, \ldots, e_{\ell-1}\}$,
then $\exp \CA = \{e_1, \ldots, e_{\ell-1}, e_\ell\}$,
where $e_\ell := | \CA \setminus \CA_X|$.
\end{lemma}

\begin{proof}
Let $\pi_X$ be an inductive factorization of $\CA_X$
and let $\exp \CA_X = \{0, e_1, \ldots, e_{\ell-1}\}$.
Let $\pi^X := \CA \setminus \CA_X = \{H_1, \ldots, H_{e_\ell}\}$ (given in any fixed order).
Then $\pi := (\pi_X, \pi^X)$ is a partition of $\CA$.
We show that $\pi$ is an inductive factorization
of $\CA$ by means of the following induction table of factorizations
starting with the inductive factorization $\pi_X$ of $\CA_X$,
see Remark \ref{rem:indfactable}.

\begin{table}[ht!b]
\renewcommand{\arraystretch}{1.5}
\begin{tabular}{lllll}\hline
$\pi_i'$ &  $\exp\CA_i'$ & ${H_i}$ & $\pi_i''$ & $\exp\CA_i''$\\ 
\hline\hline
$\pi_X$ &  $e_1, \ldots, e_{\ell-1}, 0$ &  $H_1$ & $\overline {\pi_X}$ & $e_1, \ldots, e_{\ell-1}$ \\
$\pi_X,\{H_1\}$ &  $e_1, \ldots, e_{\ell-1}, 1$ &  $H_2$ & $\overline {\pi_X}$ & $e_1 , \ldots, e_{\ell-1}$ \\
$\pi_X,\{H_1, H_2\}$ &  $e_1, \ldots, e_{\ell-1}, 2$ &  $H_3$ & $\overline {\pi_X}$ & $e_1, \ldots, e_{\ell-1}$ \\
\vdots & \vdots & \vdots & \vdots & \vdots \\ 
$\pi_X, \{H_1, H_2, \ldots, H_{e_\ell-1}\}$ &  $e_1, \ldots, e_{\ell-1}, e_\ell-1$ &  $H_{e_\ell}$ & $\overline {\pi_X}$ & $e_1, \ldots, e_{\ell-1}$ \\
$\pi = (\pi_X, \pi^X)$ &  $e_1, \ldots, e_{\ell-1}, e_\ell$ \\
\hline
\end{tabular}
\smallskip
\caption{Induction table of factorizations for $\CA$ starting at $\CA_X$}
\label{indtable:modular2}
\end{table}

It follows from Lemma \ref{lem:modular2}
applied to the consecutive triples in this induction table
that each restriction in Table \ref{indtable:modular2}
is isomorphic to $\CA_X$ and 
the induced partition $\pi_i'' = \overline {\pi_X}$ 
is in bijection with $\pi_X$. 
Thus each restriction along with the induced partition in this table  
is isomorphic to the pair $(\CA_X, \pi_X)$, and thus is 
inductively factored.
Consequently, it follows from Table \ref{indtable:modular2},
Theorem \ref{thm:add-del-factored} 
and Remark \ref{rem:indfactable} that
$\pi = (\pi_X, \pi^X)$ is an inductive factorization of $\CA$.
Observe, the notation here is consistent with the 
one introduced in Definition \ref{def:factored}(b), for  
the partition of $\CA_X$ induced from $\pi$ is just 
$\pi_X$, the one we started with.
\end{proof}

\begin{remark}
\label{rem:modular-nice}
The same argument in 
the proof of Lemma \ref{lem:modular-indfactored}
shows that if $X \in L(\CA)$ is modular of rank $r-1$
and $\CA_X$ is nice, then so is $\CA$
(albeit without any reference to exponents,  
as $\CA_X$ might not be free).
\end{remark}

As for the previous stronger freeness properties, 
inductively factored arrangements are 
compatible with the  
product construction for arrangements.

\begin{proposition}
[{\cite[Prop.\ 3.30]{hogeroehrle:factored}}]
\label{prop:product-indfactored}
Let $\CA_1, \CA_2$ be two arrangements.
Then  $\CA = \CA_1 \times \CA_2$ is inductively factored
if and only if both 
$\CA_1$ and $\CA_2$ are inductively factored.
\end{proposition}

\begin{remark}
\label{rem:modular-failure}
Observe that the implications in Lemmas 
\ref{lem:modular-indfree}, 
\ref{lem:modular-indfactored} 
and Corollary \ref{cor:modular-super}
fail without
the modularity requirement
on $X$. For let $\CA$
be a non-free $3$-arrangement and let
$X \in L(\CA)$ be of rank $2$.
Then being supersolvable, 
$\CA_X$ satisfies each of the 
stronger freeness properties, but 
$\CA$ does not.
\end{remark}

\subsection{Supersolvable Weyl arrangements}
\label{ssect:weyl}

In \cite[Cor.\ 5.15]{barakatcuntz:indfree}, 
Barakat and Cuntz 
showed that every Weyl arrangement is inductively free.
In \cite[Thm.\ 1.2]{hogeroehrle:super} and 
\cite[Cor.\ 1.4]{hogeroehrle:nice}, 
all supersolvable, respectively
inductively factored, reflection arrangements
were classified.
It is going to be useful to know all instances 
when a Weyl arrangement itself 
satisfies any of these stronger freeness properties.
It turns out that for a
Weyl arrangement $\CA = \CA(\Phi)$, 
supersolvability, niceness and inductive factoredness
all coincide.
 
\begin{theorem}
[{\cite[Thm.\ 1.2]{hogeroehrle:super}, 
\cite[Thm.\ 1.3, Cor.\ 1.4]{hogeroehrle:nice}}]
\label{thm:ssW}
Let $\Phi$ be an irreducible reduced root system with 
Weyl arrangement $\CA = \CA(\Phi)$.
Then 
\begin{itemize}
\item[(i)]
$\CA$ is supersolvable if and only if 
$\Phi$ is of type $A_n$, $B_n$, $C_n$, or $G_2$; 
\item[(ii)]
$\CA$ is inductively factored if and only if $\CA$ is supersolvable;
\item[(iii)]
$\CA$ is nice if and only if $\CA$ is inductively factored.
\end{itemize}
\end{theorem}

\subsection{Parabolic subsystems and parabolic subgroups}
\label{ssect:parabolic}

Let $\Phi$
be a reduced root system 
of rank $n$ with 
Weyl group $W$ and reflection arrangement
$\CA = \CA(\Phi) = \CA(W)$.

For $w \in W$, write 
$\Fix(w) :=\{ v\in V \mid w v = v\}$ for 
the fixed point subspace of $w$.
For $U \subseteq V$ a subspace, we 
define the \emph{parabolic subgroup}
$W_U$ of $W$ by 
$W_U := \{w \in W \mid U \subseteq \Fix(w)\}$.
By Steinberg's Theorem \cite[Thm.\ 1.5]{steinberg:invariants},
for $U \subseteq V$ a subspace, 
the parabolic subgroup
$W_U$ is itself a Coxeter group,
generated by the reflections in $W$ that are contained
in $W_U$. 
Let $X = \cap_{H \in \CA_U} H$.
Then $\CA_X = \CA_U$ and 
$X \in L(\CA)$. 
Thus, 
$\CA(W_U) = \CA_U = \CA_X = \CA(W_X)$, 
by \cite[Thm.\ 6.27, Cor.\ 6.28]{orlikterao:arrangements}.

Let $\Phi^+$
be the set of positive roots with respect to 
some set of simple roots $\Pi$  of $\Phi$.
For $\Pi_0$ a proper subset of $\Pi$, the 
\emph{(standard parabolic) subsystem} of $\Phi$
generated by $\Pi_0$ is $\Phi_0 := \BBZ \Pi_0 \cap \Phi$,
cf.~\cite[Ch.\ VI \S 1.7]{bourbaki:groupes}. 
Define $\Phi_0^+ := \Phi_0 \cap \Phi^+$,
the set of positive roots of $\Phi_0$ with respect to $\Pi_0$.
If the rank of $\Phi_0$ is $n-1$, 
then $\Phi_0$ is said to be \emph{maximal}.

Set $X_0 := \cap_{\gamma \in \Phi_0^+} H_\gamma$.
Then $\CA(\Phi)_{X_0} = \CA(\Phi_0)$. Therefore, 
the reflection arrangement $\CA(W_{X_0})$
of the parabolic subgroup $W_{X_0}$
is just $\CA(\Phi_0)$,
i.e.\ $\Phi_0$ is the root system of $W_{X_0}$.

\subsection{On the number of ideals}
\label{sect:noideals}
Suppose that $\Phi$ is irreducible with Weyl group $W$.
Let $\theta$ be the highest root in $\Phi$.
A closed formula for the number of all
ideals $\CI$ in $\Phi^+$ was given by 
Cellini and Papi 
\cite{cellinipapi:nilpotentI, cellinipapi:nilpotentII}:
\begin{equation}
\label{eq:noI}
\frac{1}{|W|} \prod_{i=1}^n(h + e_i + 1),
\end{equation}
where $h = \hgt(\theta) + 1$ is the 
Coxeter number of $\Phi$, and 
$e_1, \ldots , e_n$
are the exponents of $W$.

Following \cite{sommers:ideals}, we
call an ideal $\CI$ in $\Phi^+$ \emph{strictly positive} provided
it satisfies $\CI \cap \Pi = \varnothing$, 
i.e., provided it does not contain a simple root.
Note that this includes the empty ideal.
Sommers proved a closed formula for the number of strictly positive ideals
in \cite{sommers:ideals}:
\begin{equation}
\label{eq:noII}
\frac{1}{|W|} \prod_{i=1}^n(h + e_i - 1), 
\end{equation}
where the notation is as in  
\eqref{eq:noI}.

It is also useful to have closed expressions in terms of the rank for the 
numbers above for the classical types. 
These along with the same numbers for the exceptional types
belong to the famous Catalan sequences. For further information 
on the combinatorial aspects and ubiquity of the latter, see 
for instance \cite{stanley:book2}.

\begin{table}[ht!b]
\renewcommand{\arraystretch}{1.6}
\begin{tabular}{r|cccc}
\hline
$\Phi$ & $A_n$ & $B_n$ & $C_n$ & $D_n$  \\ 
\hline\hline
all $\CI$ 
& $\frac{1}{n+2}\binom{2n+2}{n+1}$ & $\binom{2n}{n}$ & $\binom{2n}{n}$ & 
$\binom{2n-1}{n} + \binom{2n-2}{n}$ \\
strictly positive $\CI$ 
& $\frac{1}{n+1}\binom{2n}{n}$ & $\binom{2n-1}{n-1}$ & $\binom{2n-1}{n-1}$ & 
$\binom{2n-2}{n} + \binom{2n-3}{n}$ \\
\hline
\end{tabular}
\smallskip
\caption{The number of all $\CI$ and strictly positive $\CI$ for classical $\Phi$}
\label{table:no-ideals-classical}
\end{table}

In Table \ref{table:noI} we present the 
number of ideals $\CI$ that lie in $\CI_t$ for $t \ge 1$
for the irreducible root systems of exceptional type of rank
at least $4$.
Thus the first two rows give the number of 
ideals in $\Phi^+$, respectively strictly positive  
ideals in $\Phi^+$, according to 
\eqref{eq:noI} and \eqref{eq:noII}.

\begin{table}[ht!b]
\renewcommand{\arraystretch}{1.2}
\begin{tabular}{r|rrrr}\hline
$\Phi$ &  $F_4$ & $E_6$ & $E_7$ & $E_8$\\ 
\hline\hline
$\CI$ & 105 & 833 & 4160 & 25080 \\
\hline
$\CI \subseteq \CI_2$ & 66  & 418 & 2431 & 17342\\
$\CI \subseteq \CI_3$ & 48  & 254  & 1660  & 13395 \\
$\CI \subseteq \CI_4$ & 36  & 150  & 1162  & 10714 \\
$\CI \subseteq \CI_5$ & 22  & 62  & 726  & 8330 \\
$\CI \subseteq \CI_6$ &     &     & 403  & 6623 \\
$\CI \subseteq \CI_7$ &     &     &      & 4500 \\
\hline
  $*$  & 20  & 62  & 727  & 6178 \\
\hline
\end{tabular}
\smallskip
\caption{The number of ideals $\CI$ in $\CI_t$ for $t \ge 1$}
\label{table:noI}
\end{table}

The last row in Table \ref{table:noI} labeled with $*$ gives the number of 
ideals $\CI$ that are not covered by the inductive argument of 
Theorems \ref{thm:indfree-exceptional} and \ref{thm:zeta-factors-count},
see \S \ref{ss:lastentry} and \S \ref{ss:summary},
and \S \ref{ssec:rank}.
The significance of the last entry in each column 
above the row labeled with $*$ is explained in 
\S \ref{ss:lastentry}.

\section{Arrangements of Ideal Type}
\label{s:idealtype}

We maintain the notation and setup from the Introduction and \S \ref{sect:prelims}. 
In particular, let $\Phi$ be a reduced root system in the
real $n$-space $V$ with a fixed set of simple roots $\Pi$,
so that $|\Pi| = n$, and
corresponding set of positive roots $\Phi^+$.
Throughout, let $\CA = \CA(\Phi)$ be the Weyl 
arrangement of $\Phi$.
Let $\Phi_0$ be a proper parabolic subsystem of $\Phi$,
let $\CI$ be an ideal in $\Phi^+$ and let $\CI_0 = \CI \cap \Phi_0^+$.
We start with an elementary but crucial observation.

\begin{lemma}
\label{lem:ideallocal}
Let $\Phi$, $\CI$ and $\Phi_0$ be as above.
Then, viewing $\CA_{\CI_0}$ as a subarrangement of $\CA_\CI$,
we have 
$\CA_{\CI_0} = (\CA_\CI)_{X_0}$, where 
$X_0 := \cap_{\gamma \in \Phi_0^+} H_\gamma$.
\end{lemma}

\begin{proof}
As a subarrangement of $\CA = \CA(\Phi)$,
$\CA(\Phi_0)$ coincides with $\CA_{X_0}$, 
cf.~\S \ref{ssect:parabolic}.
Therefore, 
viewing $\CA_{\CI_0}$ as a subarrangement of $\CA_\CI$, we have 
\[
\CA_{\CI_0} = \{H_\gamma \mid \gamma \in \CI_0^c = \CI^c \cap \Phi_0^+\} 
=  \CA_\CI \cap \CA_{X_0} =  (\CA_\CI)_{X_0},
\]
as desired.
Note that $X_0$ need not belong to $L(\CA_\CI)$ in general; e.g.\ see
Example \ref{ex:cond} below.
\end{proof}

The following is immediate from 
Propositions \ref{prop:indfreelocal},
\ref{prop:superlocal}, \ref{prop:indfaclocal} and Lemma \ref{lem:ideallocal}.

\begin{corollary}
\label{cor:ssideal}
Let $\Phi$, $\CI$ and $\Phi_0$ be as above.
\begin{itemize}
\item[(i)] 
If $\CA_\CI$ is supersolvable, then so is 
$\CA_{\CI_0}$.
\item[(ii)] 
If $\CA_\CI$ is inductively free, then so is 
$\CA_{\CI_0}$.
\item[(iii)] 
If $\CA_\CI$ is inductively factored, then so is 
$\CA_{\CI_0}$.
\end{itemize}
\end{corollary}

Before discussing consequences of 
Condition \ref{cond:linear}, we 
illustrate two instances in an easy example
when this condition holds, respectively fails.
Note that in this setting 
$\Phi_0$ is assumed to be a maximal parabolic subsystem of $\Phi$.

\begin{example}
\label{ex:cond}
Let $\Phi$ be of type $A_3$ with simple roots $\Pi = \{\alpha, \beta, \gamma\}$.
Let $\CI$ be the ideal in $\Phi^+$ generated by $\beta$.
Then $\CI^c = \{\alpha, \gamma\}$, so that $\CA_\CI = \{H_\alpha, H_\gamma\}$ is of rank $2$.
We consider two different maximal parabolic subsystems in turn.

(a). First let $\Phi_0$ be the subsystem of $\Phi$ of type $A_2$ generated by $\alpha$ and $\beta$.
Then $\CI_0 = \{\beta, \alpha + \beta\}$ and $\Phi^c_0 \cap \CI^c = \{\gamma\}$.
So, $\CI_0^c = \{\alpha\}$ and $\CA_{\CI_0} = \{H_\alpha\}$.
Clearly, $\CI_0$ is not an ideal in $\Phi^+$.
Note that $X_0 = H_\alpha \cap H_\beta \not \in L(\CA_\CI)$.
But $(\CA_\CI)_{X_0} = \{H_\alpha\}$ so that 
$r((\CA_\CI)_{X_0}) = r(\CA_\CI)-1$.
In particular, Condition \ref{cond:linear} holds in this instance.

(b). Second let  $\Phi_0$ be the subsystem of $\Phi$ of type 
$A_1^2$ generated by $\alpha$ and $\gamma$.
This time $\CI_0 = \varnothing$ which is of course an ideal in $\Phi^+$.
Moreover, we have $\CI_0^c = \Phi_0^+ = \{\alpha, \gamma\} = \CI^c$.
Now $X_0 = H_\alpha \cap H_\gamma \in L(\CA_\CI)$.
So that $\CA_{\CI_0} = (\CA_\CI)_{X_0} = \CA_\CI$, and in particular, 
$r(\CA_{\CI_0}) = r(\CA_\CI)$.
Here Condition \ref{cond:linear} is not satisfied, as 
$\Phi^c_0 \cap \CI^c = \varnothing$.
\end{example}

Our next observation, showing that  
Condition \ref{cond:linear} entails the presence 
of a modular element in $L(\CA_\CI)$ of rank 
$r(\CA_\CI)-1$, is pivotal for our entire analysis.

\begin{lemma}
\label{lem:condition-modular}
If $\CI \subseteq \Phi^+$ and 
$\Phi_0$ satisfy Condition \ref{cond:linear},
then for $X_0 := \cap_{\gamma \in \Phi_0^+} H_\gamma$,
the center 
\[
Z:= T((\CA_\CI)_{X_0})
\] 
of $(\CA_\CI)_{X_0}$
is modular of rank $r(\CA_\CI)-1$ in $L(\CA_\CI)$.
\end{lemma}

\begin{proof}
By construction and Lemma \ref{lem:ideallocal}, we have 
\begin{equation}
\label{eq:Z}
\CA_{\CI_0} = (\CA_\CI)_{X_0} = (\CA_\CI)_Z.
\end{equation}
Thanks to Condition \ref{cond:linear}, 
$\Phi^c_0 \cap \CI^c \not= \varnothing$ and 
$\Phi^c_0 \cap \CI^c$ is linearly ordered by height.
Consequently, there is a unique
simple root $\alpha$ in $\Phi^c_0 \cap \CI^c$.
Since $\alpha$ is linearly independent from $\CI_0^c$ and 
every root in $\CI^c$ is a $\BBZ_{\ge 0}$-linear combination 
of $\alpha$ and roots from $\CI_0^c$,
we have 
$r(\CA_{\CI_0}) = r(\CA_\CI)-1$.
We conclude that $r(Z) = r((\CA_\CI)_Z) = r(\CA_\CI)-1$, by \eqref{eq:Z}.

It remains to show that $Z$ is modular.
If $\alpha \ne \beta \in \Phi^c_0 \cap \CI^c$,
then by Condition \ref{cond:linear},
there is a $\gamma \in \Phi_0^+$ so that 
$\alpha, \beta$ and $\gamma$ are linearly dependent.
Since $\alpha$ and $\beta$ both belong to $\CI^c$, 
so does $\gamma$.
It follows that $\gamma \in \Phi_0^+ \cap \CI^c = \CI_0^c$.
Thus $H_\gamma \in \CA_{\CI_0} = (\CA_\CI)_Z$, by \eqref{eq:Z}.
On the other hand, $H_\alpha, H_\beta \in \CA_\CI \setminus (\CA_\CI)_Z$.
Since $\alpha, \beta$ and $\gamma$ are linearly dependent,
$Z$ is modular, by Lemma \ref{lem:modular1}.
\end{proof}

\begin{lemma}
\label{lem:linear}
Suppose $\Phi$, $\CI$ and $\Phi_0$ 
satisfy Condition \ref{cond:linear}.
Let $\delta$ be 
in $\Phi^c_0 \cap \CI^c$. 
Then the restriction of $\CA_\CI$ to $H_\delta$ is isomorphic to the
arrangement of ideal type $\CA_{\CI_0}$ in $\CA(\Phi_0)$.
\end{lemma} 

\begin{proof}
By Lemma \ref{lem:condition-modular},
for $X_0 := \cap_{\gamma \in \Phi_0^+} H_\gamma$,
the center
$Z:= T((\CA_\CI)_{X_0})$ of $(\CA_\CI)_{X_0}$
is modular of rank $r(\CA_\CI)-1$ in $L(\CA_\CI)$.
By \eqref{eq:Z}, 
$\CA_{\CI_0} = (\CA_\CI)_{X_0} = (\CA_\CI)_Z$.
Therefore, it follows from Lemma \ref{lem:modular2}
that the restriction map 
$\varrho_\delta : \CA_{\CI_0} \to \CA_\CI^{H_\delta}$ given by 
$H_\gamma \mapsto H_\delta \cap H_\gamma$ defines an 
isomorphism between 
$\CA_{\CI_0}$ and $\CA_\CI^{H_\delta}$.
\end{proof}

We are now in a position to prove Theorem \ref{thm:I1I}.

\begin{proof}
[Proof of Theorem \ref{thm:I1I}]
The reverse implications follow in each instance from 
Corollary \ref{cor:ssideal}.

Now consider the forward implications. 
If $\CA_\CI$ has rank at most $2$, then all statements clearly hold,
as then $\CA_\CI$  is supersolvable, 
cf.~Theorems \ref{thm:super} and \ref{thm:superindfree}
and Proposition \ref{prop:superindfactored}.
So suppose that $r(\CA_\CI) \ge 3$.

If $\Phi^c_0 \cap \CI^c = \varnothing$, 
then $\CA_\CI$ is the product of the $1$-dimensional 
empty arrangement $\varnothing_1$ and $\CA_{\CI_0}$.
So each statement follows from 
Propositions \ref{prop:product-super},
\ref{prop:product-indfree}, 
and 
\ref{prop:product-indfactored},
respectively.
Therefore, we may assume that $\Phi^c_0 \cap \CI^c \ne \varnothing$.
Then define the non-empty subarrangement 
\[
\CA_{\CI_0}^c := \{H_\beta \mid \beta \in \Phi^c_0 \cap \CI^c\}
= \CA_\CI \setminus \CA_{\CI_0}
\]
of $\CA_\CI$. We may thus decompose $\CA_\CI$ as the 
proper disjoint union 
\begin{equation*}
\CA_\CI = \CA_{\CI_0} \coprod \CA_{\CI_0}^c.
\end{equation*}
Now Condition \ref{cond:linear} implies that 
the center $Z:= T((\CA_\CI)_{X_0})$ of $(\CA_\CI)_{X_0}$
is modular of rank $r(\CA_\CI) -1$ in $L(\CA_\CI)$, 
by Lemma \ref{lem:condition-modular}.
Thus the forward implications follow 
thanks to \eqref{eq:Z} from
Corollary \ref{cor:modular-super} and  
Lemmas \ref{lem:modular-indfree} and \ref{lem:modular-indfactored},
respectively.
\end{proof}

Armed with Theorem \ref{thm:I1I}
we can now derive Theorems \ref{thm:condition}
and \ref{thm:AnBn}.

\begin{proof}[Proof of Theorem \ref{thm:condition}]
If $\CA_\CI$ is reducible, then 
$\CA_\CI$ is the product of smaller rank arrangements of ideal type.
The result then follows from the inductive hypothesis along with 
Proposition \ref{prop:product-indfree}.

If $\CA_\CI$ is irreducible and there is a 
maximal parabolic subsystem of $\Phi$
such that Condition \ref{cond:linear} is satisfied,
it follows from the inductive hypothesis and 
Theorem \ref{thm:I1I}(ii) that $\CA_\CI$ is inductively free.
\end{proof}

We now apply Theorem \ref{thm:I1I} to various types.

In \cite[Lem.\ 7.1]{sommerstymoczko}, 
Sommers and Tymoczko showed
that for $\Phi$ of type $A_n$, $B_n$ or $C_n$
and the canonical choice of maximal subsystem 
$\Phi_0$ of type 
$A_{n-1}$, $B_{n-1}$ or $C_{n-1}$, respectively, 
Condition \ref{cond:linear}
is satisfied for any ideal $\CI$ with 
$\Phi^c_0 \cap \CI^c \ne \varnothing$.

\begin{proof}[Proof of Theorem \ref{thm:AnBn}]
For $\Phi$ of rank $2$, the result follows 
by the first part of Theorem \ref{thm:super}.

Let $\Phi$ be of type $A_n$, $B_n$, or $C_n$  for $n \ge 3$ and let 
$\CA = \CA(\Phi)$ be the Weyl arrangement of $\Phi$.
We argue by induction on $n$ and suppose 
that the result holds for smaller rank 
root systems of the same type as $\Phi$.
Let $\Phi_0$ be the standard maximal parabolic subsystem of $\Phi$ of 
type $A_{n-1}$, $B_{n-1}$ or $C_{n-1}$, respectively.
If $\Phi^c_0 \cap \CI^c = \varnothing$, 
then $\CA_\CI$ is the product of the $1$-dimensional 
empty arrangement $\varnothing_1$ and an arrangement of ideal type of 
$\Phi_0$, and the latter is 
supersolvable by induction on $n$,  
and so the result follows from
Proposition \ref{prop:product-super}.

Therefore, we may assume that $\Phi^c_0 \cap \CI^c \ne \varnothing$.
It follows from \cite[\S 7]{sommerstymoczko}
that $\Phi^c_0 = \Phi^+ \setminus \Phi_0^+$
is linearly ordered by height and that 
for each $\CI$ with
$\Phi^c_0 \cap \CI^c \ne \varnothing$, 
Condition \ref{cond:linear} is satisfied.
It thus follows by induction on $n$ and 
Theorem \ref{thm:I1I}(i) that 
each $\CA_\CI$ is supersolvable.
\end{proof}

With Theorem \ref{thm:AnBn} 
we readily get further
instances of supersolvable 
arrangements of ideal type in other types as well.

\begin{example}
\label{ex:typeA}
Let $\Phi$ be a reduced root system and let 
$\CI$ be an ideal in $\Phi^+$ with  
$\CI \cap \Pi \ne \varnothing$. Suppose that the simple factors in the 
complement $\Pi \setminus \CI$ 
are all of type $A$, $B$ or $C$. Then $\CA_\CI$ is a product
of arrangements of ideal type of types $A$, $B$ or $C$.
So $\CA_\CI$ is supersolvable, thanks to 
Theorem \ref{thm:AnBn}  and Proposition \ref{prop:product-super}.
For instance this is the case for any such $\CI$ in case $\Phi$ is of type $F_4$. 

Moreover, if it is the case that $\CA_{\CI_0}$ is supersolvable 
in all instances when Condition \ref{cond:linear} holds, then 
also $\CA_\CI$ is supersolvable, by Theorem \ref{thm:I1I}(i).
This is also the case for $F_4$, as then $\Phi_0$ is of type $B_3$ or $C_3$.
Consequently, out of the total of $105$ ideals in type $F_4$ the $83$ instances  
covered in Theorem \ref{thm:indfree-exceptional} (i) and (ii) are supersolvable.
\end{example}

Here is a uniform example of a supersolvable 
arrangement of ideal type in every type.

\begin{example}
\label{ex:ht1}
Observe that $\CI_2^c = \Pi$ and so 
$\CA_{\CI_2}$ is the Boolean arrangement of rank $n$ which is
known to be supersolvable, cf.~\cite[Ex.\ 2.33]{orlikterao:arrangements}.
\end{example}

Next we consider the case when $\Phi$ is of type $D_n$.
Here and later on we use the notation for the positive roots
from \cite[Planche IV]{bourbaki:groupes}.

\begin{proof}[Proof of Theorem \ref{thm:Dn}]
Let $\Phi$ be of type $D_n$, for $n \ge 4$ and
let $\Phi_0$ be the standard subsystem of $\Phi$ of type $D_{n-1}$.
Then $\Phi^c_0 = \{ e_1 \pm e_j \mid 2 \le j \le n\}$.
Note that $\Phi^c_0$ is not linearly ordered by 
$\preceq$, for $\beta^\pm : = \beta_n^\pm : = e_1 \pm e_n$
both have height $n-1$.

We argue by induction on $n$.
For $n = 3$, $\Phi$ is of type $D_3 = A_3$,
and so the result follows from 
Theorems \ref{thm:AnBn} and \ref{thm:superindfree}. 

Now suppose that $n \ge 4$ and that the result holds for
smaller rank root systems of type $D$.
If $\Phi^c_0 \cap \CI^c = \varnothing$, 
then $\CA_\CI$ is the product of the $1$-dimensional 
empty arrangement $\varnothing_1$ and an arrangement of ideal type of 
the subsystem of type $D_{n-1}$.
Since the latter is  
inductively free by induction on the rank,  
the result follows from 
Proposition \ref{prop:product-indfree}.
Therefore, we may assume that $\Phi^c_0 \cap \CI^c \ne \varnothing$.

(a). 
Suppose first that $\CI$ is an ideal in $\Phi^+$ so that
Condition \ref{cond:linear} is satisfied.
This is precisely the case as long as not both 
of $\beta^\pm = e_1 \pm e_n$ belong to $\CI^c$.
Let $\delta$ be the unique root of maximal 
height in $\Phi^c_0 \cap \CI^c$ and set $\CI_0 = \CI \cap \Phi_0$.
Then $\CI_0$ is an ideal in $\Phi_0^+$.
Set $m_s^\CI := \hgt(\delta)$, so that
thanks to Condition \ref{cond:linear}, we have
\[
m_s^\CI = |\Phi^c_0 \cap \CI^c| = |\CA_\CI \setminus \CA_{\CI_0}|.
\]
By induction on the rank, $\CA_{\CI_0}$ is inductively free
and by construction
$m_1^\CI, \ldots, m_{k-1}^\CI$ are the non-zero 
exponents of $\CA_{\CI_0}$, 
where $m_1^\CI, \ldots, m_s^\CI$ are the ideal exponents 
of $\CI$. Note that 
here we do not partially order the ideal exponents as in \eqref{eq:idealexp}. 

It follows from Theorem \ref{thm:I1I}(ii) that 
$\CA_\CI$ is inductively free with exponents 
$\exp \CA_\CI = \{0^{n-k},  m_1^\CI, \ldots, m_s^\CI\}$,
as desired.

(b).
Now we consider the cases when 
both $\beta^\pm = e_1 \pm e_n$ do belong to 
$\CI^c$. Here we follow closely the 
proof of \cite[Thm.\ 11.1]{sommerstymoczko}.

Suppose first that $\CI$ is such that both $\beta^\pm $ are maximal in 
$\Phi^c_0 \cap \CI^c$ with respect to $\preceq$. 
Set $\CI^+ := \CI \cup \{\beta^+\}$.
Then by case (a) proved above,
$\CA_{\CI^+}$ is inductively free
and 
\[
\exp \CA_{\CI^+} = \{m_1^\CI, \ldots, m_{n-2}^\CI, n-1, n-2\}.
\]
Let
$(\CA_\CI, \CA_\CI' = \CA_{\CI^+}, \CA_\CI'')$ be
the triple of $\CA_\CI$ with respect to $H_{\beta^+}$.
It remains to show that 
$\CA_\CI^{H_{\beta^+}}$ is inductively free with
\[
\exp \CA_\CI^{H_{\beta^+}} = \{m_1^\CI, \ldots, m_{n-2}^\CI, n-1\}.
\]
It then follows from Theorem \ref{thm:add-del} 
that $\CA_\CI$ is inductively free with 
\[
\exp \CA_\CI = \{m_1^\CI, \ldots, m_{n-2}^\CI, n-1, n-1\}.
\]
In order to show that $\CA_\CI^{H_{\beta^+}}$ is inductively free with
the desired exponents, we argue as follows.
Consider the triple of $\CA_\CI^{H_{\beta^+}}$ 
with respect to $H_{\beta^+} \cap H_{\beta^-}$.
Then the deleted arrangement
$\CA_\CI^{H_{\beta^+}}\setminus \{H_{\beta^+} \cap H_{\beta^-}\}$
coincides with $\CA_{\CI^-}^{H_{\beta^+}}$, where 
$\CI^- := \CI \cup \{\beta^-\}$ and so
the latter is inductively free
with exponents 
$\{m_1^\CI, \ldots, m_{n-2}^\CI, n-2\}$,
by case (a) above.
For $\CI^-$ satisfies Condition \ref{cond:linear} above and 
$\beta^+$ is maximal in $\Phi^c_0 \cap (\CI^-)^c$.
Finally, we need to show that the restricted arrangement
$\left(\CA_\CI^{H_{\beta^+}}\right)^{H_{\beta^+} \cap H_{\beta^-}}$
is inductively free. 
Since $H_{\beta^+} \cap H_{\beta^-}$ 
coincides with the intersection 
of the null spaces of $e_1$ and $e_n$, 
arguing as in the proof of \cite[Thm.\ 11.1]{sommerstymoczko},
this restricted arrangement coincides
with an arrangement of ideal type in 
a root system of type $B_{n-2}$
with exponents given by
$\{m_1^\CI, \ldots, m_{n-2}^\CI\}$.
Thanks to 
Theorems \ref{thm:AnBn} and \ref{thm:superindfree},
the latter is inductively free.
So, $\CA_\CI^{H_{\beta^+}}$ is inductively free,
by Theorem \ref{thm:add-del}. 

Finally, we consider the 
case when $\delta = e_1 + e_{2n - 1 - l}$ 
is the maximal element of $\Phi^c_0 \cap \CI^c$ for some $l > n - 1$.
We argue as in the
proof of \cite[Thm.\ 11.1]{sommerstymoczko}
and deduce again that $\CA_\CI$ is 
inductively free with
\[
\exp \CA_\CI = \{m_1^\CI, \ldots, m_{n-2}^\CI, n-1, l\},
\]
as follows.
By induction on $|\CI^c|$, we know that $\CA_{\CI'}$ is 
inductively free, for $\CI' = \CI \cup \{\delta\}$, with
\[
\exp \CA_{\CI'} = \{m_1^\CI, \ldots, m_{n-2}^\CI, n-1, l-1\}.
\] 
By the argument from the proof of \cite[Thm.\ 11.1]{sommerstymoczko},
the restricted arrangement 
$\CA_\CI^{H_\delta}$ is 
isomorphic to 
$\CA_\CI^{H_{\beta^+}}$ 
in the case above, where both $\beta^\pm$ were maximal in 
$\Phi^c_0 \cap \CI^c$. 
From that case we
infer that $\CA_\CI^{H_{\beta^+}}$ 
is inductively free, thus so is $\CA_\CI^{H_\delta}$ with 
\[
\exp \CA_\CI^{H_\delta} = \{m_1^\CI, \ldots, m_{n-2}^\CI, n-1\}.
\]
It follows from Theorem \ref{thm:add-del} that also in this case 
$\CA_\CI$ is inductively free.
\end{proof}

\begin{example}
\label{ex:dn}
Theorem \ref{thm:I1I}(i) can also be used to 
show that there are arrangements of ideal type 
which are supersolvable in type 
$D_n$ beyond the cases treated in Theorem \ref{thm:d4ss}(i).
For instance let $\CI$ be the ideal generated by
$e_{n-3} + e_n$,
$e_{n-3} - e_n$, or  $e_{n-2} + e_{n-1}$.
Let $\Phi_0$ be the standard 
subsystem of type $A_{n-1}$ so that 
the generator of $\CI$ does not belong to $\Phi_0$
(there are two choices in the last instance).
Then one easily checks that 
Condition \ref{cond:linear} is satisfied. 
The fact that $\CA_\CI$ is supersolvable then follows 
from Theorems \ref{thm:AnBn} and \ref{thm:I1I}(i).
\end{example}

Next we give a uniform argument for the penultimate ideal 
$\CI_{h-1} = \{\theta\}$ in all cases.

\begin{proof}[Proof of Theorem \ref{thm:penultimate}]
Let $\CA = \CA(W)$ be the Weyl arrangement of $W$.
Consider the triple 
$(\CA, \CA', \CA'')$
with respect to $H_\theta$. 
In case there is only one root length
$W$ is transitive on $\Phi$. 
In the other instances 
(i.e.\ for $B_n$ and $F_4$)
the restrictions of $\CA$ to hyperplanes with respect 
to a short and long root are isomorphic,
cf.~\cite[Prop.\ 6.82, Table C.9]{orlikterao:arrangements}.
Therefore, since $\CA$ is inductively free, 
$(\CA, \CA', \CA'')$ is a triple of inductively free arrangements.
In particular, 
$\CA' = \CA_{\CI_{h-1}}$ is inductively free.
\end{proof}

\begin{remark}
\label{rem:linearexceptional}
Lemma \ref{lem:linear} and Theorem \ref{thm:I1I}(ii)
can also be employed in the exceptional instances as well.
Suppose $\CI$ is an ideal in $\Phi^+$ such that
Condition \ref{cond:linear} is fulfilled
for some fixed choice of a maximal parabolic subsystem $\Phi_0$.
Let $\delta$ be the unique root of maximal 
height in $\Phi^c_0 \cap \CI^c$
and $\CI_0 = \CI \cap \Phi_0^+$.
Consider the restriction of 
$\CA_\CI$ with respect to $H_\delta$.
In our applications in Section \ref{S:indfree} 
we argue by induction 
on the rank of the underlying root system, 
so that the arrangement of ideal type $\CA_{\CI_0}$ 
of smaller rank is inductively free.
Then using Lemma \ref{lem:linear}, 
Theorems \ref{thm:AnBn} and \ref{thm:Dn}, it follows that 
$\CA_\CI^{H_\delta} \cong \CA_{\CI_0}$ is inductively 
free. It thus follows from 
Theorem \ref{thm:I1I}(ii)
that also $\CA_\CI$ is inductively free with the desired exponents.
We illustrate this in the following example. 
\end{remark}

\begin{example}
\label{ex:f4}
Let $\Phi$ be of type $F_4$.
We use the notation for the roots in $\Phi$ as in 
\cite[Planche VIII]{bourbaki:groupes}. 
We consider two examples of ideals.
In our first case 
let $\CI$ be the ideal generated by $1120$ and $0122$.
Let $\Phi_0$ be the standard subsystem of type $C_3$.
Then 
$\CI^c = \{\underline{1000}, 0100, 0010, 0001, 
\underline{1100},0110, 0011, 
\underline{1110}, 
0120, 0111, \underline{1111}, 0121\}$,
where the roots in 
$\Phi^c_0 \cap \CI^c$ are underlined.
It is easy to check that $\CI$ and $\Phi_0$ 
satisfy Condition \ref{cond:linear}.

In our second example
let $\CI$ be the ideal generated by $0121$.
This time let $\Phi_0$ be the standard subsystem of type $B_3$.
Then 
$\CI^c = \{1000, 0100, 0010, \underline{0001}, 
1100,0110, \underline{0011}, 1110$, 
$0120, \underline{0111}, 1120, \underline{1111}, 1220\}$,
where the roots in 
$\Phi^c_0 \cap \CI^c$ are underlined.
It is again 
easy to check that $\CI$ and $\Phi_0$ 
satisfy Condition \ref{cond:linear}.

Here the subsystems used must not be interchanged;  
for $\Phi_0$ of type $B_3$, the first ideal
does not satisfy Condition \ref{cond:linear}
and likewise  neither does the second 
for $\Phi_0$ of type $C_3$.

It follows from Theorems \ref{thm:AnBn} and 
\ref{thm:I1I}(i) that both $\CA_\CI$ are supersolvable
and that the exponents are $\{1,3,4,4\}$ and  
$\{1,3,4,5\}$, respectively; cf.~Example \ref{ex:typeA}.
\end{example}

We investigate the inductively free $\CA_\CI$ in 
the exceptional instances in more  
detail in \S \ref{S:indfree}.

\section{Inductively free $\CA_\CI$ for $\Phi$ of exceptional type}
\label{S:indfree}

\subsection{}
\label{ss:41}
Thanks to Proposition \ref{prop:product-indfree}, 
Theorem \ref{thm:indfree-exceptional} 
readily reduces to the case when $\Phi$ is irreducible
which we assume from now on.
Let $\CA = \CA(\Phi)$ be the Weyl arrangement 
for $\Phi$ irreducible of exceptional type. 
Since any arrangement of rank at most two is inductively free, 
we may suppose that $W$ has rank at least $4$.

We use the labeling of the Dynkin diagram of $\Phi$ and the 
notation 
for roots in $\Phi$ from \cite[Planche V - VIII]{bourbaki:groupes}.
We argue by induction on the rank of the underlying 
root system and therefore assume
that each arrangement of ideal type
is inductively free for root systems of smaller rank.
So fix $\Phi$ and let $\CI$ be an ideal in $\Phi^+$. 
Arguing further by
induction on $|\CI^c|$, we may assume that 
$\CA_{\CJ}$ is inductively free for every ideal $\CJ$
properly containing $\CI$. 

\subsection{}
\label{ss:numberIt}
Our strategy is to consider all ideals $\CI$ such
that $\CI \subseteq \CI_t$ but $\CI \not \subseteq \CI_{t+1}$
for successive values of $t \ge 1$.
This means each such ideal contains a root of height $t$
but no roots of smaller height.
Then we determine all instances when 
Condition \ref{cond:linear}
is satisfied for each such $\CI$ for a suitable 
choice of subsystem $\Phi_0$ of $\Phi$ 
in Tables \ref{table:height3} -- \ref{table:height6}
below. In each of these tables  
we list for a given root $\beta$ in $\CI$ of height $t$
a parabolic subsystem $\Phi_0$ of $\Phi$ and the resulting 
set of roots in $\Phi^c_0 \cap \CI^c$ 
relevant for Condition \ref{cond:linear}.
Here we determine $\Phi^c_0 \cap \CI^c$
under the assumption that 
$\beta$ is the only root 
of height $t$ in $\CI$.
In case there are additional roots 
of height $t$ in $\CI$, 
the set $\Phi^c_0 \cap \CI^c$
may be smaller but still satisfies 
Condition \ref{cond:linear}.
In each case we know by 
induction that $\CA_{\CI_0}$ is inductively 
free, so that we can conclude from 
Theorem \ref{thm:I1I}(ii) that $\CA_\CI$
is also inductively free.

For instance, 
consider the next to last entry for $E_7$ in 
Table \ref{table:height3}.
Here 
$\CI = \left\langle\stackrel{001110}{_{0\ }}\right\rangle$,
and for $\Phi_0$ of type $E_6$ it follows that 
$\CI_0 = \left\langle\stackrel{00111}{_{0}}\right\rangle$.
The fact that  $\CA_{\CI_0}$ is inductively 
free follows from the last row for $E_6$ of the same table.
So $\CA_\CI$ is inductively free by induction
and Theorem \ref{thm:I1I}(ii).
For another example, consider the last case for 
for $E_7$ in Table \ref{table:height3}.
Here 
$\CI = \left\langle\stackrel{000111}{_{0\ }}\right\rangle$,
and for $\Phi_0$ of type $E_6$ we have 
$\CI_0^c = \Phi_0^+$, so that 
$\CA_{\CI_0} = \CA(E_6)$ which is inductively free.
So then again 
$\CA_\CI$ is inductively free thanks to Theorem \ref{thm:I1I}(ii).

\subsection{}
\label{ss:lastentry}
Returning to Table \ref{table:noI}, 
the last entry in each column 
above the last row labeled by $*$
indicates that for every ideal $\CI$ with 
$\CI \subseteq \CI_t$, 
there is no maximal parabolic subsystem 
$\Phi_0$ so that 
Condition \ref{cond:linear} is satisfied.
Naturally, if $\CI$ is small, then $\CI^c$ is large.
So that for sufficiently small $\CI$ 
Condition \ref{cond:linear} fails simply because
$\Phi^c_0 \cap \CI^c$ is no longer linear for any 
choice of maximal subsystem $\Phi_0$.
So for instance, if $\Phi$ is of type $F_4$,
then this is the case for all
$\CI$ whose roots have height at
least $5$, and according to 
the entry in Table \ref{table:noI}, there
are exactly $22$ such instances.
The final row labeled by $*$ 
gives the total number of all 
ideals $\CI$ where 
Condition \ref{cond:linear} fails for 
every choice of a maximal subsystem $\Phi_0$
of $\Phi$. 
From the list of these cases we have 
removed the $2$ instances corresponding 
to the full Weyl arrangement $\CA_{\CI_h} =\CA(\Phi)$
and the case of the penultimate ideal $\CA_{\CI_{h-1}}$
which are both inductively free.

\subsection{}
\label{ss:notI2}
Now  if $\CI \cap \Pi \neq \varnothing$, i.e.\ if $\CI$ is not
strictly positive, then $\CI^c$ is either 
the complement of an ideal in a 
smaller rank root system or is the product of complements 
of ideals in direct products 
of smaller rank root systems. It therefore 
follows from our induction hypothesis 
and Proposition \ref{prop:product-indfree}
that also $\CA_\CI$ is inductively free in this instance.
We may therefore assume that $\CI$ is strictly positive, 
i.e.\ $\CI \subseteq \CI_2$.
Observe that $\CI_2^c = \Pi$ and so 
$\CA_{\CI_2}$ is inductively free, 
see Example \ref{ex:ht1}.

\subsection{}
\label{ss:I2notI3}
Next suppose that $\CI \subseteq \CI_2$ and 
$\CI \not\subseteq \CI_3$, i.e.\ $\CI$ contains a root of height $2$.
Then $\CA_\CI$ is again the product of 
two arrangements of ideal type of smaller rank root systems
and so the result follows from our induction hypothesis 
and Proposition \ref{prop:product-indfree}.
We therefore may assume that $\CI \subseteq \CI_3$.

%For instance, let $\Phi$ be of type $F_4$ and let 
%$\CI = \left\langle 1100 \right\rangle$.
%Then $\CA_\CI$ is the product of the two Weyl arrangements 
%$\CA(A_1)$ and $\CA(C_3)$.

\subsection{}
\label{ss:I3notI4}
Next we consider the case when 
$\CI \subseteq \CI_3$ and $\CI \not\subseteq \CI_4$.
Then there is a root of height $3$ in $\CI$
but no root of smaller height.
One readily checks that there is always a suitable maximal
subsystem $\Phi_0$ in each case so that 
Condition \ref{cond:linear} is satisfied.
We list the various cases in Table \ref{table:height3}
below. 
For each fixed root $\beta$ of height $3$ we consider 
the case $\CI = \langle \beta \rangle$ and 
list a suitable subsystem $\Phi_0$ such that Condition \ref{cond:linear} 
is fulfilled. This is easily checked in each instance.
In type $E_6$ there are two standard subsystems of type $D_5$ which are 
conjugate by the graph automorphism of $E_6$.
In order to distinguish between them, here and later on, 
$D_5$, respectively $D_5'$, denotes the 
one complementary to $\alpha_1$, respectively $\alpha_6$. 

\begin{table}[ht!b]
\renewcommand{\arraystretch}{1.6}
\begin{tabular}{r|c|c|l}\hline
$\Phi$ &  $\beta \in \CI$ & $\Phi_0$ & $\Phi^c_0 \cap \CI^c$ \\
\hline\hline
$F_4$ 
& $1110$ &  $C_3$ & $1000, 1100$ \\
& $0120$ &  $C_3$ & $1000, 1100, 1110, 1111$  \\
& $0111$ &  $B_3$ & $0001, 0011$ \\
\hline
$E_6$ 
&  $\stackrel{11100}{_{0}}$ & $D_5$ &  $ \stackrel{10000}{_{0}}, \stackrel{11000}{_{0}} $ \\
& $\stackrel{01100}{_{1}}$ & $A_5$ &  $ \stackrel{00000}{_{1}}, \stackrel{00100}{_{1}}, \stackrel{00110}{_{1}}, \stackrel{00111}{_{1}} $ \\
& $\stackrel{01110}{_{0}}$ & $D_5$ &  $ \stackrel{10000}{_{0}}, \stackrel{11000}{_{0}}, \stackrel{11100}{_{0}}, \stackrel{11100}{_{1}} $ \\
& $\stackrel{00110}{_{1}}$ & $A_5$  &  $ \stackrel{00000}{_{1}}, \stackrel{00100}{_{1}}, \stackrel{01100}{_{1}}, \stackrel{11100}{_{1}} $ \\
& $\stackrel{00111}{_{0}}$ & $D_5'$ & $ \stackrel{00001}{_{0}}, \stackrel{00011}{_{0}} $ \\
\hline
$E_7$ 
&  $\stackrel{111000}{_{0\ }}$ & $D_6$ &  $ \stackrel{100000}{_{0\ }}, \stackrel{110000}{_{0\ }} $ \\
& $\stackrel{011000}{_{1\ }}$ & $A_6$ &  $ \stackrel{000000}{_{1\ }}, \stackrel{001000}{_{1\ }}, \stackrel{001100}{_{1\ }}, \stackrel{001110}{_{1\ }}, \stackrel{001111}{_{1\ }} $ \\
& $\stackrel{011100}{_{0\ }}$ & $D_6$ &  $ \stackrel{100000}{_{0\ }}, \stackrel{110000}{_{0\ }}, \stackrel{111000}{_{0\ }}, \stackrel{111000}{_{1\ }} $ \\
& $\stackrel{001100}{_{1\ }}$ & $A_6$  &  $ \stackrel{000000}{_{1\ }}, \stackrel{001000}{_{1\ }}, \stackrel{011000}{_{1\ }}, \stackrel{111000}{_{1\ }} $ \\
& $\stackrel{001110}{_{0\ }}$ & $E_6$ & $ \stackrel{000001}{_{0\ }}, \stackrel{000011}{_{0\ }}, \stackrel{000111}{_{0\ }} $ \\
& $\stackrel{000111}{_{0\ }}$ & $E_6$ & $ \stackrel{000001}{_{0\ }}, \stackrel{000011}{_{0\ }} $ \\
\hline
$E_8$
&  $\stackrel{1110000}{_{0\ \ }}$ & $D_7$ &  $ \stackrel{1000000}{_{0\ \ }}, \stackrel{1100000}{_{0\ \ }} $ \\
& $\stackrel{0110000}{_{1\ \ }}$ & $A_7$ &  $ \stackrel{0000000}{_{1\ \ }}, \stackrel{0010000}{_{1\ \ }}, \stackrel{0011000}{_{1\ \ }}, \stackrel{0011100}{_{1\ \ }}, \stackrel{0011110}{_{1\ \ }}, \stackrel{0011111}{_{1\ \ }} $ \\
& $\stackrel{0111000}{_{0\ \ }}$ & $D_7$ &  $ \stackrel{1000000}{_{0\ \ }}, \stackrel{1100000}{_{0\ \ }}, \stackrel{1110000}{_{0\ \ }}, \stackrel{1110000}{_{1\ \ }} $ \\
& $\stackrel{0011000}{_{1\ \ }}$ & $A_7$  &  $ \stackrel{0000000}{_{1\ \ }}, \stackrel{0010000}{_{1\ \ }}, \stackrel{0110000}{_{1\ \ }}, \stackrel{1110000}{_{1\ \ }} $ \\
& $\stackrel{0011100}{_{0\ \ }}$ & $E_7$ & $ \stackrel{0000001}{_{0\ \ }}, \stackrel{0000011}{_{0\ \ }}, \stackrel{0000111}{_{0\ \ }}, \stackrel{0001111}{_{0\ \ }} $ \\
& $\stackrel{0001110}{_{0\ \ }}$ & $E_7$ & $ \stackrel{0000001}{_{0\ \ }}, \stackrel{0000011}{_{0\ \ }}, \stackrel{0000111}{_{0\ \ }} $ \\
& $\stackrel{0000111}{_{0\ \ }}$ & $E_7$ & $ \stackrel{0000001}{_{0\ \ }}, \stackrel{0000011}{_{0\ \ }} $ \\
\hline
\end{tabular}
\smallskip
\caption{Condition \ref{cond:linear} for 
$\CI \subseteq \CI_3$ and $\CI \not\subseteq \CI_4$}
\label{table:height3}
\end{table}

We should point out that in case $\CI$ contains additional roots
of height $3$, then Condition \ref{cond:linear}  is still fulfilled
albeit $\Phi^c_0 \cap \CI^c$ may be a proper subset of 
the one shown in Table \ref{table:height3}.
For instance, if $\Phi$ is of type $E_6$ and 
$\CI = \left\langle\stackrel{01100}{_{1}}, \stackrel{00110}{_{1}}\right\rangle$,
then for $\Phi_0$ of type $A_5$ as in 
the second row for $E_6$ in Table \ref{table:height3}, 
$\Phi^c_0 \cap \CI^c$ only consists of 
$\alpha_2$ and $\alpha_2 + \alpha_4$.

It follows by induction on the rank and the results from 
\S\S \ref{ss:notI2} - \ref{ss:I3notI4} that $\CA_\CI$ is 
inductively free for any $\CI$ with $\CI \not\subseteq \CI_4$.

\subsection{}
\label{ss:I4notI5}
Next, we consider the case when 
$\CI \subseteq \CI_4$ and $\CI \not\subseteq \CI_5$.
Then there is a root of height $4$ in $\CI$
but no root of smaller height.
In Table \ref{table:height4},
for each fixed root $\beta$ of height $4$ we first consider 
the case $\CI = \langle \beta \rangle$ and --
provided there exists one -- 
list a suitable maximal parabolic 
subsystem $\Phi_0$ such that Condition \ref{cond:linear} holds.
This is then easy to verify in each instance.

\begin{table}[ht!b]
\renewcommand{\arraystretch}{1.6}
\begin{tabular}{r|c|c|l}\hline
$\Phi$ &  $\beta \in \CI$ & $\Phi_0$ & $\Phi^c_0 \cap \CI^c$ \\
\hline\hline
$F_4$ 
& $1111$ &  $C_3$ & $1000, 1100, 1110, 1120, 1220$ \\
& $1120$ &  $C_3$ & $1000, 1100, 1110, 1111$  \\
& $0121$ &  $B_3$ & $0001, 0011, 0111, 1111$ \\
\hline
$E_6$ 
&  $\stackrel{11110}{_{0}}$ & $D_5$ &  $ \stackrel{10000}{_{0}}, \stackrel{11000}{_{0}}, \stackrel{11100}{_{0}}, \stackrel{11100}{_{1}} $ \\
& $\stackrel{11100}{_{1}}$ & $D_5$ &  $ \stackrel{10000}{_{0}}, \stackrel{11000}{_{0}}, \stackrel{11100}{_{0}}, \stackrel{11110}{_{0}}, \stackrel{11111}{_{0}} $ \\
& $\stackrel{01110}{_{1}}$ & $\times$ &  $\times$ \\
& $\stackrel{00111}{_{1}}$ & $D_5'$  &  $ \stackrel{00001}{_{0}}, \stackrel{00011}{_{0}}, \stackrel{00111}{_{0}}, \stackrel{01111}{_{0}}, \stackrel{11111}{_{0}} $ \\
& $\stackrel{01111}{_{0}}$ & $D_5'$ & $ \stackrel{00001}{_{0}}, \stackrel{00011}{_{0}}, \stackrel{00111}{_{0}}, \stackrel{00111}{_{1}} $ \\
\hline
$E_7$ 
&  $\stackrel{111100}{_{0\ }}$ & $D_6$ &  $ \stackrel{100000}{_{0\ }}, \stackrel{110000}{_{0\ }}, \stackrel{111000}{_{0\ }}, \stackrel{111000}{_{1\ }} $ \\
& $\stackrel{111000}{_{1\ }}$ & $D_6$ &  $ \stackrel{100000}{_{0\ }}, \stackrel{110000}{_{0\ }}, \stackrel{111000}{_{0\ }}, \stackrel{111100}{_{0\ }}, \stackrel{111110}{_{0\ }}, \stackrel{111111}{_{0\ }} $ \\
& $\stackrel{011100}{_{1\ }}$ & $\times$ &  $\times$ \\
& $\stackrel{011110}{_{0\ }}$ & $E_6$ & $ \stackrel{000001}{_{0\ }}, \stackrel{000011}{_{0\ }}, \stackrel{000111}{_{0\ }}, \stackrel{001111}{_{0\ }}, \stackrel{001111}{_{1\ }} $ \\
& $\stackrel{001110}{_{1\ }}$ & $E_6$  & $ \stackrel{000001}{_{0\ }}, \stackrel{000011}{_{0\ }}, \stackrel{000111}{_{0\ }}, \stackrel{001111}{_{0\ }}, \stackrel{011111}{_{0\ }}, \stackrel{111111}{_{0\ }} $ \\  
& $\stackrel{001111}{_{0\ }}$ & $E_6$ & $ \stackrel{000001}{_{0\ }}, \stackrel{000011}{_{0\ }}, \stackrel{000111}{_{0\ }} $ \\
\hline
$E_8$
&  $\stackrel{1111000}{_{0\ \ }}$ & 
$D_7$ &  $ \stackrel{1000000}{_{0\ \ }}, \stackrel{1100000}{_{0\ \ }}, \stackrel{1110000}{_{0\ \ }}, \stackrel{1110000}{_{1\ \ }} $ \\
& $\stackrel{1110000}{_{1\ \ }}$ &
$D_7$ &  $ \stackrel{1000000}{_{0\ \ }}, \stackrel{1100000}{_{0\ \ }}, \stackrel{1110000}{_{0\ \ }}, \stackrel{1111000}{_{0\ \ }}, \stackrel{1111100}{_{0\ \ }}, \stackrel{1111110}{_{0\ \ }}, \stackrel{1111111}{_{0\ \ }} $ \\
& $\stackrel{0111000}{_{1\ \ }}$ & $\times$ &  $\times$ \\
& $\stackrel{0011100}{_{1\ \ }}$ & 
$E_7$ & $ \stackrel{0000001}{_{0\ \ }}, \stackrel{0000011}{_{0\ \ }}, \stackrel{0000111}{_{0\ \ }}, \stackrel{0001111}{_{0\ \ }}, \stackrel{0011111}{_{0\ \ }}, \stackrel{0111111}{_{0\ \ }}, \stackrel{1111111}{_{0\ \ }} $ \\
& $\stackrel{0111100}{_{0\ \ }}$ & 
$E_7$ & $ \stackrel{0000001}{_{0\ \ }}, \stackrel{0000011}{_{0\ \ }}, \stackrel{0000111}{_{0\ \ }}, \stackrel{0001111}{_{0\ \ }}, \stackrel{0011111}{_{0\ \ }}, \stackrel{0011111}{_{1\ \ }} $ \\
& $\stackrel{0011110}{_{0\ \ }}$ & $E_7$ & $ \stackrel{0000001}{_{0\ \ }}, \stackrel{0000011}{_{0\ \ }}, \stackrel{0000111}{_{0\ \ }}, \stackrel{0001111}{_{0\ \ }} $ \\
& $\stackrel{0001111}{_{0\ \ }}$ & $E_7$ & $ \stackrel{0000001}{_{0\ \ }}, \stackrel{0000011}{_{0\ \ }}, \stackrel{0000111}{_{0\ \ }} $ \\
\hline
\end{tabular}
\smallskip
\caption{Condition \ref{cond:linear} for 
$\CI \subseteq \CI_4$ and $\CI \not\subseteq \CI_5$}
\label{table:height4}
\end{table}

There are some ideals $\CI$ of this kind
when there is no maximal standard
parabolic subsystem $\Phi_0$ such that Condition \ref{cond:linear}
is fulfilled. This is indicated with the label 
``$\times$'' in the corresponding row.
For instance, this occurs when $\Phi$ is of type $E_6$ and 
$\CI = \left\langle\stackrel{01110}{_{1}}\right\rangle$.
If there are additional roots of height $4$ in $\CI$, then 
we can ensure again that Condition \ref{cond:linear} holds.
E.g., let 
$\CI = \left\langle\stackrel{11110}{_{0}}, \stackrel{01110}{_{1}}\right\rangle$.
Then for $\Phi_0$ of type $D_5$, 
$\Phi^c_0 \cap \CI^c$ is the same set as in the first row for $E_6$
in Table \ref{table:height4}. 
One checks that 
$\CI = \left\langle\stackrel{01110}{_{1}}\right\rangle$
and 
$\CI = \left\langle\stackrel{01110}{_{1}}, \stackrel{11111}{_{0}}\right\rangle$
are the only ideals of this nature for which 
Condition \ref{cond:linear} fails for any choice of 
a maximal standard subsystem $\Phi_0$ of $\Phi$.
Likewise, for $\Phi$ of type $E_7$,
respectively of type $E_8$,  
there are $3$, respectively $8$, ideals of this kind
for which Condition \ref{cond:linear} does not hold for any choice of 
a maximal standard parabolic subsystem $\Phi_0$ of $\Phi$.
For instance, the ones in type $E_7$ are 
$\CI = \left\langle\stackrel{011100}{_{1\ }}\right\rangle$, 
$\CI = \left\langle\stackrel{011100}{_{1\ }}, \stackrel{111110}{_{0\ }}\right\rangle$,
and
$\CI = \left\langle\stackrel{011100}{_{1\ }}, \stackrel{111111}{_{0\ }}\right\rangle$.

\subsection{}
\label{ss:I5}
Next, we consider the case when $\CI \subseteq \CI_5$.
In case of $F_4$ and $E_6$ one readily checks that 
there is no maximal standard parabolic subsystem $\Phi_0$ such that
Condition \ref{cond:linear} is satisfied.

It follows from Remark \ref{rem:linearexceptional}, 
Theorem \ref{thm:penultimate}, 
\cite[Cor.\ 5.15]{barakatcuntz:indfree}, and 
the data in Tables \ref{table:height3} and \ref{table:height4}, 
that for $F_4$ and $E_6$, there are at most 
$20$, respectively $62$, $\CA_\CI$
which might fail to be inductively free. 
This number is listed in the last row of 
Table \ref{table:noI} labelled by $*$.

\subsection{}
\label{ss:I5notI6}
We continue by 
considering the case when 
$\CI \subseteq \CI_5$ and $\CI \not\subseteq \CI_6$
for $E_7$ and $E_8$.
Then there is a root of height $5$ in $\CI$
but no root of smaller height.
We list all cases when 
Condition \ref{cond:linear} is satisfied in 
Table \ref{table:height5} below.
The notation is as in the previous tables.
In particular, as before, ``$\times$'' indicates
that Condition \ref{cond:linear} fails for any choice of 
maximal parabolic subsystem.

\begin{table}[ht!b]
\renewcommand{\arraystretch}{1.6}
\begin{tabular}{r|c|c|l}\hline
$\Phi$ &  $\beta \in \CI$ & $\Phi_0$ & $\Phi^c_0 \cap \CI^c$ \\
\hline\hline
$E_7$ 
&  $\stackrel{111110}{_{0\ }}$ & $\times$ &  $\times$ \\
& $\stackrel{111100}{_{1\ }}$ & $\times$ &  $\times$ \\
& $\stackrel{011110}{_{1\ }}$ & $\times$ &  $\times$ \\
& $\stackrel{012100}{_{1\ }}$ & $\times$ &  $\times$ \\
& $\stackrel{011111}{_{0\ }}$ & $E_6$ & $ \stackrel{000001}{_{0\ }}, \stackrel{000011}{_{0\ }}, \stackrel{000111}{_{0\ }}, \stackrel{001111}{_{0\ }}, \stackrel{001111}{_{1\ }} $ \\
& $\stackrel{001111}{_{1\ }}$ & $E_6$  & $ \stackrel{000001}{_{0\ }}, \stackrel{000011}{_{0\ }}, \stackrel{000111}{_{0\ }}, \stackrel{001111}{_{0\ }}, \stackrel{011111}{_{0\ }}, \stackrel{111111}{_{0\ }} $ \\  
\hline
$E_8$
&  $\stackrel{1111100}{_{0\ \ }}$ & $\times$ &  $\times$ \\
& $\stackrel{1111000}{_{1\ \ }}$ & $\times$ &  $\times$ \\
& $\stackrel{0111100}{_{1\ \ }}$ & $\times$ &  $\times$ \\
& $\stackrel{0121000}{_{1\ \ }}$ & $\times$ &  $\times$ \\
& $\stackrel{0011110}{_{1\ \ }}$ & 
$E_7$ & $ \stackrel{0000001}{_{0\ \ }}, \stackrel{0000011}{_{0\ \ }}, \stackrel{0000111}{_{0\ \ }}, \stackrel{0001111}{_{0\ \ }}, \stackrel{0011111}{_{0\ \ }}, \stackrel{0111111}{_{0\ \ }}, \stackrel{1111111}{_{0\ \ }} $ \\
& $\stackrel{0111110}{_{0\ \ }}$ & 
$E_7$ & $ \stackrel{0000001}{_{0\ \ }}, \stackrel{0000011}{_{0\ \ }}, \stackrel{0000111}{_{0\ \ }}, \stackrel{0001111}{_{0\ \ }}, \stackrel{0011111}{_{0\ \ }}, \stackrel{0011111}{_{1\ \ }} $ \\
& $\stackrel{0011111}{_{0\ \ }}$ & 
$E_7$ & $ \stackrel{0000001}{_{0\ \ }}, \stackrel{0000011}{_{0\ \ }}, \stackrel{0000111}{_{0\ \ }}, \stackrel{0001111}{_{0\ \ }} $ \\
\hline
\end{tabular}
\smallskip
\caption{Condition \ref{cond:linear} for 
$\CI \subseteq \CI_5$ and $\CI \not\subseteq \CI_6$}
\label{table:height5}
\end{table}

If $\CI$ is an ideal as above,  
containing more than one root $\beta$ of height
$5$ so that 
already for the smaller ideal $\langle \beta \rangle$ 
there is no 
parabolic subsystem $\Phi_0$ which satisfies Condition \ref{cond:linear},
then it is readily checked that 
this is also the case for the larger ideal $\CI$.
On the other hand if $\CI$ is an ideal which contains 
roots $\beta$ of height $5$ 
so that Condition \ref{cond:linear} holds 
for the ideal $\langle \beta \rangle$ for 
some such $\beta$ and for some choice of  
maximal parabolic subsystem $\Phi_0$, 
then Condition \ref{cond:linear} also holds 
for the larger ideal $\CI$ and the same $\Phi_0$.

For instance, consider the ideal 
$\CI = \left\langle\stackrel{111110}{_{0\ }}, \stackrel{001111}{_{1\ }}\right\rangle$
(cf.~the first and last rows of the cases for 
$E_7$ in Table \ref{table:height5}).
We see that for $\Phi_0$ of type $E_6$ 
Condition \ref{cond:linear} holds and 
$\Phi^c_0 \cap \CI^c = 
\left\{\stackrel{000001}{_{0\ }}, \stackrel{000011}{_{0\ }}, 
\stackrel{000111}{_{0\ }}, \stackrel{001111}{_{0\ }}, \stackrel{011111}{_{0\ }}\right\}$.

For $\Phi$ of type $E_7$,
respectively of type $E_8$,  
there are $323$, respectively $355$, ideals of this kind
for which Condition \ref{cond:linear} does not hold for any choice of 
a maximal standard parabolic subsystem $\Phi_0$ of $\Phi$.

\subsection{}
\label{ss:I6}
Next, we consider the case when $\CI \subseteq \CI_6$.
One checks that in case of $E_7$ 
there is no maximal parabolic subsystem $\Phi_0$ such that
Condition \ref{cond:linear} is satisfied.
For $E_8$ there is only one case of such an ideal $\CI$.
This is listed in Table \ref{table:height6} below.
Note that for $E_8$ one checks that if $\CI \subseteq \CI_7$,
then again there is no maximal standard parabolic subsystem $\Phi_0$ such that
Condition \ref{cond:linear} is satisfied.

For $\Phi$ of type $E_7$,
respectively of type $E_8$,  
there are $403$, respectively $1317$, ideals of this kind
for which Condition \ref{cond:linear} does not hold for any choice of 
a maximal standard parabolic subsystem $\Phi_0$ of $\Phi$.

\begin{table}[ht!b]
\renewcommand{\arraystretch}{1.6}
\begin{tabular}{r|c|c|l}\hline
$\Phi$ &  $\beta \in \CI$ & $\Phi_0$ & $\Phi^c_0 \cap \CI^c$ \\
\hline\hline
$E_8$
&  $\stackrel{1111110}{_{0\ \ }}$ & $\times$ &  $\times$ \\
& $\stackrel{1111100}{_{1\ \ }}$ & $\times$ &  $\times$ \\
& $\stackrel{0111110}{_{1\ \ }}$ & $\times$ &  $\times$ \\
& $\stackrel{1121000}{_{1\ \ }}$ & $\times$ &  $\times$ \\
& $\stackrel{0121100}{_{1\ \ }}$ & $\times$ &  $\times$ \\
& $\stackrel{0111111}{_{0\ \ }}$ & 
$E_7$ & $\stackrel{0000001}{_{0\ \ }}, \stackrel{0000011}{_{0\ \ }}, 
\stackrel{0000111}{_{0\ \ }}, \stackrel{0001111}{_{0\ \ }}, 
\stackrel{0011111}{_{0\ \ }}, \stackrel{0011111}{_{1\ \ }}$ \\
\hline
\end{tabular}
\smallskip
\caption{Condition \ref{cond:linear} for 
$\CI \subseteq \CI_6$ and $\CI \not\subseteq \CI_7$}
\label{table:height6}
\end{table}

\subsection{}
\label{ss:summary}
Finally, by induction on the rank and thanks to 
Theorem \ref{thm:penultimate} and 
\cite[Cor.\ 5.15]{barakatcuntz:indfree} along with 
the results from 
\S \S \ref{ss:notI2} -- \ref{ss:I6}, including  
the data in Tables \ref{table:height3} -- \ref{table:height6}, 
it follows 
that for $E_7$, respectively $E_8$, there are at most 
$403 + 323 + 3 - 2 = 727$, 
respectively $4500 + 1317 + 355 +8 - 2 = 6178$ 
arrangements of ideal type $\CA_\CI$
which might fail to be inductively free. 
This number of instances is listed in the last row of 
Table \ref{table:noI} labelled by $*$. 

Finally, 
\cite[Cor.\ 5.15]{barakatcuntz:indfree} along with 
Theorem \ref{thm:penultimate} complete the proof 
Theorem \ref{thm:indfree-exceptional}. 

\begin{remark}
\label{rem:computed}
Above we have dealt with all instances
when Condition \ref{cond:linear} 
holds in the exceptional types.
As observed this leads to a surprisingly small 
number of ideals left to be considered further. 
This number is shown in the last row of Table  \ref{table:noI} above
labeled with $*$.
Here we have already taken the cases 
$\CI_{h-1} = \{\theta\}$ and $\CI_h = \varnothing$ into account as well
which give inductively free 
arrangements of ideal type, by  
Theorem \ref{thm:penultimate} and 
\cite[Cor.\ 5.15]{barakatcuntz:indfree}.

For each of these remaining instances 
other than the outstanding $6178$ for type $E_8$, T.~Hoge was able to 
check that $\CA_\CI$ 
is inductively free by computational means.
\end{remark}

Finally Theorem \ref{thm:indfree} follows from 
Theorem \ref{thm:indfree-exceptional} and Remark \ref{rem:computed}.

\section{The Poincar\'e polynomial  $\CI(t)$ of $\CI$}
\label{sec:poincare}

\subsection{Rank-generating functions of posets of regions of real arrangements}
\label{s:rankgenerating}

In this subsection let $\CA$ be a 
hyperplane arrangement in the real vector space $V=\BBR^\ell$. 
A \emph{region} of $\CA$ is a connected component of the 
complement $M(\CA) := V \setminus \cup_{H \in \CA}H$ of $\CA$.
Let $\RR := \RR(\CA)$ be the set of regions of $\CA$.
For $R, R' \in \RR$, we let $\CS(R,R')$ denote the 
set of hyperplanes in $\CA$ separating $R$ and $R'$.
Then with respect to a choice of a fixed 
base region $B$ in $\RR$, we can partially order
$\RR$ as follows:
\[
R \le R' \quad \text{ if } \quad \CS(B,R) \subseteq \CS(B,R').
\]
Endowed with this partial order, we call $\RR$ the
\emph{poset of regions of $\CA$ (with respect to $B$)} and denote it by
$P(\CA, B)$. This is a ranked poset of finite rank,
where $\rk(R) := |\CS(B,R)|$, for $R$ a region of $\CA$, 
\cite[Prop.\ 1.1]{edelman:regions}.
The \emph{rank-generating function} of $P(\CA, B)$ is 
defined to be the following polynomial in 
$\BBZ_{\ge 0}[t]$
\begin{equation*}
\label{eq:rankgen}
\zeta(P(\CA,B), t) := \sum_{R \in \RR}t^{\rk(R)}. 
\end{equation*}
This poset along with its rank-generating function
was introduced by Edelman 
\cite{edelman:regions}.

If $\CA$ is the product of arrangements $\CA = \CA_1 \times \CA_2$
then so is the rank-generating function
of its poset of regions
\begin{equation}
\label{eq:posregprod}
\zeta(P(\CA,B), t) = \zeta(P(\CA_1,B_1), t) \cdot \zeta(P(\CA_2,B_2), t),
\end{equation}
where $B = B_1 \times B_2$, see \cite[\S 3.3]{stanley:book}.

The following theorem due to Jambu and Paris, 
\cite[Prop.\ 3.4, Thm.\ 6.1]{jambuparis:factored},
was first shown by Bj\"orner, Edelman and Ziegler
for $\CA$ supersolvable in \cite[Thm.\ 4.4]{bjoerneredelmanziegler}.

\begin{theorem}
\label{thm:mult-zeta}
If $\CA$ is inductively factored, then
there is a suitable choice of a base region $B$ so that 
$\zeta(P(\CA,B), t)$ satisfies the multiplicative formula
\begin{equation}
\label{eq:poinprod}
\zeta(P(\CA,B), t) = \prod_{i=1}^\ell (1 + t + \ldots + t^{e_i}),
\end{equation}
where $\{e_1, \ldots, e_\ell\} = \exp \CA$ is the 
set of exponents of $\CA$.
\end{theorem}

We recall the inductive construction from the proof of 
\cite[Thm.\ 4.4]{bjoerneredelmanziegler}.
Let $\CA$ be a real arrangement of rank $r$.
Suppose that $X \in L(\CA)$ is modular of rank $r-1$.
Let $\CA_0 := \CA_X$. 
Then there is an order-preserving, surjective map
\[
\pi : P(\CA, B) \to P(\CA_0, \pi(B)), 
\]
which is induced by inclusion of regions.
For $R \in \RR$, let $\FF(R) := \pi\inverse(\pi(R))$ 
be its \emph{fibre} under $\pi$.
We say that $B$ in $\RR$ is \emph{canonical (with respect to $\CA_0$)} 
provided $\FF(B)$ is linearly ordered within $P(\CA,B)$.
The proof of the following lemma is just the argument from 
\cite[Thm.\ 4.4]{bjoerneredelmanziegler}.

\begin{lemma}
\label{lem:modular-zeta}
Let $\CA$ be a real arrangement of rank $r$.
Suppose that $X \in L(\CA)$ is modular of rank $r-1$.
Let $\CA_0 := \CA_X$ and set $e_\ell := | \CA \setminus \CA_0|$. 
Then there is a 
region $B$ in $\RR$ which is canonical with respect to $\CA_0$. 
In addition,  setting $B_0 := \pi(B)$, we have the following:
\begin{itemize}
\item[(i)]  $\zeta(P(\CA,B), t) = \zeta(P(\CA_0,B_0), t) \cdot (1 + t + \ldots + t^{e_\ell})$.
\item[(ii)]
Suppose that 
there are non-negative integers $e_1, \ldots, e_{\ell-1}$ such that  
$\zeta(P(\CA_0,B_0), t)$ 
factors as in \eqref{eq:poinprod}.
Then $\zeta(P(\CA,B), t)$ 
also factors as in \eqref{eq:poinprod}
(with $e_\ell = | \CA \setminus \CA_0|$).
\end{itemize}
\end{lemma}

\begin{proof}
Since $X \in L(\CA)$ is modular of rank $r-1$,
any two 
hyperplanes in $\CA \setminus \CA_0$ 
do not intersect in 
a region of $\RR(\CA_0)$, 
by Lemma \ref{lem:modular1}. It follows that 
for any $R \in \RR$ the regions in $\FF(R)$  
are linearly ordered by adjacency and  
each fibre has the same length 
$e_\ell = | \CA \setminus \CA_0|$.
Fix a region $B$ in $\FF(R)$
at the end of such a linear chain of adjacencies
for some $R \in \CR$.
Then $B$  is canonical with respect to $\CA_0$.
Moreover, for $R \in \RR$, calculating the rank 
in $P(\CA,B)$ with respect to $B$, 
we have  
$\rk(R) = \rk_0(\pi(R)) + \rk_{\FF(R)}(R)$,
where $\rk_0(\pi(R))$ is the rank of $\pi(R)$ in 
$P(\CA_0,B_0)$ with respect to $B_0$ and 
$\rk_{\FF(R)}(R)$ is the rank of 
$R$ in its fibre which is induced from the rank in $P(\CA,B)$.
Hence the rank-generating function 
$\zeta(P(\CA,B), t)$ 
of the poset of regions $P(\CA,B)$ with respect to $B$ 
is the product of 
the rank-generating function 
$\zeta(P(\CA_0,B_0), t)$ 
of the poset of regions $P(\CA_0,B_0)$ with respect to $B_0$
and the rank-generating function for each fibre.
Thus, since $B$ is canonical with respect to $\CA_0$, (i) follows.

Finally, (ii) follows directly from (i).
\end{proof}

The relevance of Lemma \ref{lem:modular-zeta} stems from the
fact that it can be used to show that the rank generating function 
$\zeta(P(\CA,B), t)$ factors as in \eqref{eq:poinprod}
even when $\CA_0$ is not inductively factored 
(or supersolvable). We illustrate this with an example from 
the setting from 
\S \ref{s:idealtype} where we also use the notation from that section.

\begin{example}
\label{ex:modular-zeta}
Let $n \ge 5$ and $\Phi$ be of type $D_n$
and let $\Phi_0$ be the standard subsystem of type $D_{n-1}$.
Let $\CI$ be the ideal in $\Phi^+$ generated 
by $e_1 - e_j$
for $2 \le j \le n$.
Then Condition \ref{cond:linear} holds.
Note that with the notation from \S \ref{s:idealtype}, we have
$\CI_0 = \varnothing$, so that $\CA_{\CI_0} = \CA(D_{n-1})$
is the reflection arrangement of $\Phi_0$.
In particular, $\CA_{\CI_0}$ is not supersolvable and  
not inductively factored.
For $B$ a region of $\CA_\CI$ containing the fundamental dominant 
Weyl chamber of $\Phi$ and for $\CA_0 := \CA_{\CI_0}$ and 
$B_0 = \pi(B)$ as above, we see from \eqref{eq:ai} that 
$\zeta(P(\CA_0,B_0), t)$ is just 
the Poincar\'e polynomial of $W(D_{n-1})$ which factors
according to \eqref{eq:solomon}.
Therefore, by Lemma \ref{lem:modular-zeta}(ii), 
$\zeta(P(\CA_\CI,B), t)$ 
factors as in \eqref{eq:poinprod} as well.
As a consequence of \eqref{eq:ai}, 
Conjecture \ref{conj:st1} holds in these instances.
\end{example}

\subsection{The Poincar\'e polynomial  $\CI(t)$ and the rank-generating function of the poset of regions of $\CA_\CI$}
\label{ssec:rank}

There is a connection between the 
combinatorics of an 
arrangement of ideal type $\CA_\CI$, by way of 
the Poincar\'e polynomial $\CI(t)$ 
of $\CI$ and 
the geometry of the arrangement $\CA_\CI$,  
via the theory of 
rank-generating functions of the poset of regions of $\CA_\CI$,
as follows.

\begin{remarks}
\label{rem:weyl-regions}

(i).
In \cite[\S 12]{sommerstymoczko}, 
Sommers and Tymoczko observed that 
%if $\CA_\CI$ is of ideal type and 
if we fix a region $B$ of the set of 
regions $\RR = \RR(\CA_\CI)$ of $\CA_\CI$ which contains the dominant 
Weyl chamber of $\Phi$, then thanks to \cite[Prop.\ 6.1]{sommerstymoczko},
the rank-generating function of the poset of regions $P(\CA_\CI,B)$ of 
$\CA_\CI$ with respect to $B$ 
is just the Poincar\'e polynomial $\CI(t)$ of 
$\CI$ introduced in \eqref{eq:st}: 
\begin{equation}
\label{eq:ai}
\zeta(P(\CA_\CI,B), t) 
= \sum_{R \in \RR} t^{\rk(R)} = \sum_{S \in \CW^I} t^{|S|} = \CI(t).
\end{equation}

(ii).
Thanks to Theorem \ref{thm:mult-zeta} and \eqref{eq:ai},
the Poincar\'e polynomial $\CI(t)$ of a 
supersolvable or inductively factored 
arrangement of ideal type $\CA_\CI$ 
satisfies \eqref{eq:st1} and thus for all 
supersolvable or inductively factored $\CA_\CI$ 
Conjecture \ref{conj:st1} holds.
Thus Theorem \ref{thm:d4ss} gives further 
evidence for this conjecture. 

(iii).
Theorem \ref{thm:st1-classical} follows from 
Theorems \ref{thm:AnBn} and \ref{thm:mult-zeta} 
combined with \eqref{eq:ai}.
\end{remarks}
 
In view of Remark \ref{rem:weyl-regions}(i), 
Theorem \ref{thm:zeta-factors} is an easy consequence of 
Lemmas \ref{lem:condition-modular} and \ref{lem:modular-zeta}.

\begin{proof}[Proof of Theorem \ref{thm:zeta-factors}]
Thanks to Remark \ref{rem:weyl-regions}(i), it suffices
to consider 
the rank-generating function  $\zeta(P(\CA_\CI,B), t)$ of the 
poset of regions of $\CA_\CI$
for a region $B$ of $\RR(\CA_\CI)$ containing 
the dominant Weyl chamber of $\Phi$.
If $\CA_\CI$ is reducible, then $\CA_\CI$ is the product
of two  arrangements of ideal type of smaller rank root systems.
The multiplicative property for $\zeta(P(\CA_\CI,B), t)$
then follows from \eqref{eq:posregprod}
and the inductive hypothesis made in the statement.

So suppose that $\CA_\CI$ is irreducible 
and that Condition \ref{cond:linear} holds.
It then follows from the inductive hypothesis and 
Lemmas \ref{lem:condition-modular} and \ref{lem:modular-zeta}(ii) that 
$\zeta(P(\CA_\CI,B), t)$ factors as in 
\eqref{eq:poinprod}, as claimed.
\end{proof}

Theorem \ref{thm:zeta-factors-count} follows equally readily.

\begin{proof}[Proof of Theorem \ref{thm:zeta-factors-count}]
In cases (i) and (ii), we argue as in 
the proof of Theorem \ref{thm:zeta-factors}.

If $\CI = \{\theta\}$ or $\CI = \varnothing$,
then the result follows from Theorem \ref{thm:st1-penultimate}
and Solomon's formula \eqref{eq:solomon}, respectively.

Because the very same instances have already been accounted for
in Table \ref{table:cond} in connection with 
our analysis of the inductively free $\CA_\CI$ 
for $\Phi$ of exceptional type in 
Theorem \ref{thm:indfree-exceptional},
the final statement follows from the determination of all the
instances covered in (i) and (ii) from Section \ref{S:indfree}.
\end{proof}

\begin{remarks}
\label{rem:non-mult-zeta}

(i).
The relevance in studying the class of inductively factored real arrangements
for our purpose lies in the fact that it is the largest class 
of free arrangements 
which satisfies the multiplicative formula \eqref{eq:poinprod}.
This formula does not hold in general 
for the larger class of inductively free arrangements.
For instance, the real simplicial arrangement 
``$A_4(17)$'' from Gr\"unbaum's list \cite{gruenbaum}
does not factor as in 
\eqref{eq:poinprod}, as observed by Terao, see 
\cite[p.~277]{bjoerneredelmanziegler}.
T.~Hoge has checked that this arrangement is 
inductively free.

Therefore, in the context of 
arrangements of ideal type,   
where we can't expect more than 
inductive freeness in general,
according to our discussion above,
the general factorization property from 
Theorem \ref{thm:zeta-factors}
(predicted for all $\CA_\CI$ by 
Conjecture \ref{conj:st1})
shows that the arrangements
of ideal type are rather special
in this regard.

(ii).
The multiplicative formula \eqref{eq:poinprod}
for a free real arrangement $\CA$ 
involving the exponents frequently plays a crucial 
role in related contexts. 

For instance, for a fixed $w$ in $W$, define 
its Poincar\'e polynomial by 
$P_w(t) = \sum_{x \le w}t^{\ell(x)}$,
where $\le$ is the Bruhat-Chevalley order on $W$.
E.g., for $w_0$ the longest word in $W$, 
$P_{w_0}(t) = W(t)$, see \eqref{eq:poncarecoxeter}.
In geometric terms, $P_w(t^2)$ is the 
Poincar\'e polynomial of the Schubert variety $X(w) = \overline{BwB}/B$
in the full flag manifold  $G/B$ 
of the semisimple Lie group $G$ 
associated with $w$ in the Weyl group $W$ of $G$.
 
Consider the subarrangement $\CA_w$ of the Weyl arrangement
$\CA(W)$ given by the hyperplanes corresponding to the 
set of roots 
$N(w)$ introduced in 
\eqref{eq:Nw}.
According to a criterion due to Carrell and Peterson
\cite{carrell}, 
the Schubert variety $X(w)$ is rationally smooth
if and only if $P_w(t)$ is palindromic.

All instances are known when $P_w(t)$
admits a factorization analogous to \eqref{eq:solomon}.
This is the case precisely when $X(w)$ is rationally smooth.
Equivalently, this is the case precisely when $\CA_w$ is free and  
the exponents then make an appearance in  
this analogue of \eqref{eq:solomon} for $P_w(t)$, 
cf.~\cite[Thm.\ 3.3]{slofstra:schubert}.
This is also the case precisely when the
rank-generating function of the 
poset of regions of $\CA_w$ coincides
with $P_w(t)$, \cite{ohyoo}.
\end{remarks}

\section{Supersolvable and inductively factored $\CA_\CI$}
\label{S:ssAI}

In this final section we 
prove Theorems \ref{thm:d4notss} and \ref{thm:d4ss}.
First we consider the situation for $\Phi$ of type $D_4$. 

\begin{lemma}
\label{lem:d4}
Let $\Phi$ be of type $D_4$. Then $h = 6$.
Then all arrangements of ideal type $\CA_\CI$
but the three corresponding to 
$\CI_6 = \varnothing$, $\CI_5$, and $\CI_4$
are supersolvable 
and all but the ones corresponding to 
$\CI_6 = \varnothing$ and $\CI_5$,
are inductively factored.
Thus of the $50$ arrangements of ideal type $\CA_\CI$ in type $D_4$,
$47$ are supersolvable and $48$ are inductively factored.
\end{lemma}
 
\begin{proof}
According to 
Table \ref{table:no-ideals-classical}, 
there are $50$ arrangements of ideal type in this case and 
$20$ strictly positive ideals.
Since each proper parabolic subsystem of $D_4$ is a product of 
type $A$ subsystems, 
it follows from Theorem \ref{thm:AnBn} and 
Proposition \ref{prop:product-super}
(see Example \ref{ex:typeA})
that the $30$ arrangements $\CA_\CI$, 
where $\CI$ is not strictly positive are all supersolvable.

Next, consider all $\CI$ with $\CI \subseteq \CI_2$ and 
$\CI \not\subseteq \CI_3$. Then there is a root of height $2$ in $\CI$.
Again, since $\CA_\CI$ is a product of smaller rank arrangements 
of ideal type, it follows from Theorem \ref{thm:AnBn} and 
Proposition \ref{prop:product-super} that all 
$\CA_\CI$ of this kind are supersolvable. 

Next suppose that $\CI \subseteq \CI_3$ and 
$\CI \not\subseteq \CI_4$.
Then there is a root of height $3$ in $\CI$.
There are $7$ such ideals. 
One checks that there is always a rank $3$ subsystem $\Phi_0$
of type $A_3$, so that Condition \ref{cond:linear} is satisfied.
Thus, by Theorems \ref{thm:AnBn} and \ref{thm:I1I}(i),
each such arrangement $\CA_\CI$ is supersolvable.

Summarizing the observations above, 
$\CA_\CI$ is supersolvable
for any $\CI \not\subseteq \CI_4$, thanks to 
Theorems \ref{thm:AnBn} and \ref{thm:I1I}(i)
and Proposition \ref{prop:product-super}.
This applies to $47$ out of the $50$ ideals.
The $3$ ideals left to consider are
$\CI_4 = \{e_1 + e_3, e_1 + e_2\}$, 
$\CI_5 = \{e_1 + e_2\}$, 
and $\CI_6 = \varnothing$.
Of course, $\CA_{\CI_6} = \CA(D_4)$ is neither supersolvable,
nor inductively factored,  
e.g.\ see \cite[Ex.\ 5.5]{jambuterao:free} or 
Theorem \ref{thm:ssW}.
One checks directly that $\CA_{\CI_4}$ is not
supersolvable. 
In addition,
one can show that $\CA_{\CI_5}$ is not nice, so in particular, 
is not inductively factored, nor supersolvable.
In contrast, for $\CI_4$
one can show that
\begin{equation}
\label{eq:d4fac}
\{H_{e_2-e_3}\}, 
\{H_{e_3-e_4}, H_{e_2 -e_4}, H_{e_2 + e_3}\},
\{H_{e_1-e_2}, H_{e_1 -e_3}, H_{e_1 - e_4}\},
\{H_{e_3+e_4}, H_{e_2 +e_4}, H_{e_1 + e_4}\}
\end{equation}
is an inductive factorization of $\CA_{\CI_4}$, 
where the notation is as in \cite[Planche IV]{bourbaki:groupes}.
We show this directly as follows.
Let $\delta = e_2 + e_3$.
By the argument above, for $\CI = \langle \delta \rangle$
we have $\CA_\CI$ is supersolvable with exponents
$\exp \CA_\CI = \{ 1,2,3,3\}$.
Note that $\CI_4^c\setminus \{\delta\} = \CI^c$.
Consider the triple of $\CA_{\CI_4}$ with respect to $H_\delta$.
Then we have $( \CA_{\CI_4}, \CA_{\CI_4}' = \CA_\CI, \CA_{\CI_4}'')$.
Since $\CA_{\CI_4}' = \CA_\CI$ is supersolvable, it admits a  
canonical inductive factorization given by 
\[
\{H_{e_2-e_3}\}, 
\{H_{e_3-e_4}, H_{e_2 -e_4}\},
\{H_{e_1-e_2}, H_{e_1 -e_3}, H_{e_1 - e_4}\},
\{H_{e_3+e_4}, H_{e_2 +e_4}, H_{e_1 + e_4}\},
\]
stemming from a maximal chain of modular elements in $L(\CA_\CI)$, 
see Proposition \ref{prop:superindfactored}, see also 
\cite[Ex.\ 2.4]{terao:factored} and \cite[Prop.\ 3.11]{hogeroehrle:factored}.
Moreover, one checks that 
$\CA_{\CI_4}''$ is inductively factored with inductive factorization
given by  
\[
\{\bar H_{e_2-e_3}\}, 
\{\bar H_{e_1-e_2}, \bar H_{e_1 -e_3}, \bar H_{e_1 - e_4}\},
\{\bar H_{e_3+e_4}, \bar H_{e_2 +e_4}, \bar H_{e_1 + e_4}\}.
\]
It follows from Theorem \ref{thm:add-del-factored} that 
the partition of $\CA_\CI$ given in \eqref{eq:d4fac} 
defines an inductive factorization of $\CA_{\CI_4}$.

Thus precisely $47$ of the $50$
arrangements $\CA_\CI$ are supersolvable and 
precisely $48$ are inductively factored. 
\end{proof}

\begin{corollary}
\label{cor:d4-zeta}
Conjecture \ref{conj:st1}
holds for $\Phi$ of type $D_4$.
\end{corollary}

\begin{proof}
Let $\Phi$ be of type $D_4$ and 
let $\CI$ be an ideal in $\Phi^+$.
It follows from Lemma \ref{lem:d4} that either $\CA_\CI$ is 
inductively factored, or else $\CI = \{\theta\}$, or $\CI = \varnothing$.
So the result follows from 
Theorem \ref{thm:mult-zeta} and \eqref{eq:ai},
Theorem \ref{thm:st1-penultimate}, and \eqref{eq:solomon}, respectively.
\end{proof}

\begin{remark}
\label{rem:st1-dn}
We note that one can show that out of all $\binom{2n-1}{n} + \binom{2n-2}{n}$
ideals in type $D_n$, there are just $2^{n-2}$ ideals 
for which Condition \ref{cond:linear} fails 
with respect to $\Phi_0$ being the standard subsystem of type $D_{n-1}$.
However, if $\CI$ is the largest such ideal, i.e.\ $\CI = \langle e_{n-2} + e_{n-1}\rangle$,
then $\CA_\CI$ is supersolvable, see Example \ref{ex:dn}.
Thus Conjecture \ref{conj:st1} holds in this case, by
Theorem \ref{thm:mult-zeta} and \eqref{eq:ai}.
Likewise, for $\CI = \{\theta\}$, or $\CI = \varnothing$,
Conjecture \ref{conj:st1} holds, 
by Theorem \ref{thm:st1-penultimate} and \eqref{eq:solomon}.

Thus, assuming the factorization result 
from Conjecture \ref{conj:st1} for $D_{n-1}$, 
it follows from Theorem \ref{thm:zeta-factors}
that it holds for all instances in $D_n$ as well
with at most $2^{n-2}-3$ exceptions.

A.\ Schauenburg was able to check that 
Conjecture \ref{conj:st1} always holds also in these 
additional 
$2^{n-2}-3$ cases for $5 \le n \le 7$ by direct computational means.
\end{remark}

We are now in a position to prove Theorems \ref{thm:d4notss} and \ref{thm:d4ss}.

\begin{proof}
[Proof of Theorem \ref{thm:d4notss}]
We prove both parts simultaneously.
For $\Phi$ of type $D_4$, the result follows from 
Lemma \ref{lem:d4}.
Now let $\Phi$ be as in the statement of rank at least $5$
and let $\Phi_0$ be the standard subsystem of $\Phi$ of type $D_4$.
Since, $\CI_0 = \CI \cap \Phi_0^+$ also consists of roots of height at least $4$,
respectively $5$, by the result for $D_4$, $\CA_{\CI_0}$ is not supersolvable, 
respectively not inductively factored.
It now follows from  Corollary \ref{cor:ssideal}(i) and (iii)
that the same holds for $\CA_\CI$.
\end{proof}

\begin{remark}
\label{rem:d4notnice}
We can strengthen Theorem \ref{thm:d4notss}(ii) slightly as follows.
Thanks to Lemma \ref{lem:d4}, $\CA_\CI$ is not nice for 
$\CI \subseteq \CI_5$.
This also holds for $\Phi$ as in the statement of 
Theorem \ref{thm:d4notss}, 
arguing as in the proof above and using
Remark \ref{rem:factored} and Lemma \ref{lem:ideallocal}. 
\end{remark}

\begin{proof}
[Proof of Theorem \ref{thm:d4ss}]

(i).
We argue by induction on the rank.
For $\Phi$ of type $D_4$, the result follows from Lemma \ref{lem:d4}.
Now let $\Phi$ be of 
type $D_n$ for $n \ge 5$, $E_6$, $E_7$, or $E_8$ and 
suppose that the result
is true for smaller rank root systems of these kinds.

If $\CI$ contains a root of height $1$ or $2$,
then it follows that $\CA_\CI$ is the product of 
arrangements of ideal type for smaller rank root systems
of types $A$, $D$ or $E$ and 
each ideal corresponding to a factor still contains all roots of 
height $3$ of the smaller rank root system.
Thus by induction on the rank, Theorem \ref{thm:AnBn}, and 
Proposition \ref{prop:product-super}, the product $\CA_\CI$ is also 
supersolvable. 

It remains to consider the case when $\CI = \CI_3$.
So assume $\CI = \CI_3$ 
and let $\Phi_0$ be the standard subsystem of $\Phi$ of type $D_{n-1}$,
$D_5$, $D_6$, or $D_7$, respectively.
Then Condition \ref{cond:linear} is 
satisfied. Therefore, since 
$\CA_{\CI_0}$ is supersolvable by induction 
(for $\CI_0$ is the ideal $\Phi_0^+ \cap \CI_3$ in $\Phi_0$), 
so is $\CA_\CI$, thanks to Theorem \ref{thm:I1I}(i).

(ii).
Again, we argue by induction on the rank.
For $\Phi$ of type $D_4$, the result follows 
again from Lemma \ref{lem:d4}.
As above, if $\CI$ also contains a root of height $1$ or $2$, then 
$\CA_\CI$ is the product of 
arrangements of ideal type for smaller rank root systems
of types $A$, $D$ or $E$ and 
each ideal corresponding to a factor still contains all roots of 
height $4$ of the smaller rank root system.
Thus by induction on the rank, Theorem \ref{thm:AnBn}, and 
Proposition \ref{prop:product-indfactored}, the product $\CA_\CI$ is also 
inductively factored. 

Now suppose that $\CI \subseteq \CI_3$ but $\CI \not\subseteq \CI_4$, 
so that $\CI$ still contains a root of height $3$ but none of 
smaller height.

Suppose first that $\Phi$ is of type $D_n$ for $n \ge 5$.
If $\CI$ contains $e_i - e_{i+3}$ for some $1 \le i \le n-3$ or 
$e_{n-3} + e_n$, then for $\Phi_0$ the standard subsystem of
type $D_{n-1}$, Condition \ref{cond:linear} holds.
Therefore, since $\CI_0$ still contains all roots of height 
$4$ in $\Phi_0^+$, it follows from induction on $n$ that 
$\CA_{\CI_0}$ is inductively factored. Thus so is
$\CA_\CI$, thanks to Theorem \ref{thm:I1I}(iii).
Finally, if the only root of height $3$ in $\CI$ is $e_{n-2} + e_{n-1}$,
then let $\Phi_0$ be one of the standard subsystems of 
type $A_{n-1}$ in $\Phi$. 
Again it is easy to see that then Condition \ref{cond:linear} 
is satisfied. Thus, thanks to 
Theorem \ref{thm:AnBn}, 
$\CA_{\CI_0}$ is supersolvable and 
so is $\CA_\CI$, according to Theorem \ref{thm:I1I}(i),
and so $\CA_\CI$ is inductively factored, 
by Proposition \ref{prop:superindfactored}.

Now suppose $\Phi$ is of type $E$. Since $\CI$ admits a root of 
height $3$, it follows from Table \ref{table:height3} that there always
exists a maximal parabolic subsystem $\Phi_0$ so that 
Condition \ref{cond:linear} holds.
Since $\Phi_0$ is of type $A$, $D$ or $E$,
and $\CI_0$ still contains all the roots of height $4$ of 
$\Phi_0^+$, so that 
by induction on the rank, Theorem \ref{thm:AnBn}
and the case for type $D_n$ just proved, 
$\CA_{\CI_0}$ is inductively factored, 
then so is $\CA_\CI$, according to  
Theorem \ref{thm:I1I}(iii).

It remains to consider the case when $\CI = \CI_4$.
So let $\CI = \CI_4$.
Suppose first that $\Phi$ is of type $D_n$ for $n \ge 5$.
Since $\CI$ contains $e_1 - e_5$, 
Condition \ref{cond:linear} is valid 
for $\Phi_0$ the standard subsystem of type $D_{n-1}$.
Since $\CI_0$ is the ideal in $\Phi_0^+$ consisting of 
all roots of height at least $4$, it follows 
from induction on $n$ that 
$\CA_{\CI_0}$ is inductively factored, and thus so is
$\CA_\CI$, thanks to Theorem \ref{thm:I1I}(iii).

Finally, for $\Phi$ of type $E$ and $\CI = \CI_4$, 
it follows from Table \ref{table:height4}
there always
exists a maximal parabolic subsystem $\Phi_0$ so that 
Condition \ref{cond:linear} holds.
Since $\Phi_0$ is of type $A$, $D$ or $E$,
and $\CI_0$ is the ideal in $\Phi_0^+$ consisting of 
all roots of height at least $4$, 
it follows from 
induction on the rank, Theorem \ref{thm:AnBn}
and the case for type $D_n$ just proved
that $\CA_{\CI_0}$ is inductively factored, and 
once again, so is $\CA_\CI$, 
owing to Theorem \ref{thm:I1I}(iii).
\end{proof}

%%%%%%%%%%%%%%%%%%%%%%%%%%%%%%%%%%%%%%%%%%%%%%%%%%%%%%%%%%%%%%%%%%%%%%
%%%%%%%%%%%%% Acknowledgments
%%%%%%%%%%%%%%%%%%%%%%%%%%%%%%%%%%%%%%%%%%%%%%%%%%%%%%%%%%%%%%%%%%%%%%

\bigskip {\bf Acknowledgments}: 
The research of this work was supported by 
DFG-grant RO 1072/16-1.

I am grateful to T.~Hoge for checking that 
the possible exceptions in Table \ref{table:cond} 
in type $F_4$ (20 cases), $E_6$ (62 cases) 
and $E_7$ (727 cases) are all inductively free,
as predicted by Conjecture  \ref{conj:indfree}. 
He also checked that the simplicial arrangement 
``$A_4(17)$'' from Gr\"unbaum's list 
is inductively free.

Thanks are also due to 
A.~Schauenburg for computing 
the data in Table \ref{table:noI} 
for the cases when $\CI \subseteq \CI_t$ for $t \ge 3$
and for checking that 
Conjecture \ref{conj:st1} holds for $D_n$ for $5 \le n \le 7$
and for $E_7$.

%%%%%%%%%%%%%%%%%%%%%%%%%%%%%%%%%%%%%%%%%%%%%%%%%%%%%%%%%%%%%%%%%%%%%%
%%%%%%%%%%%%% bibliography
%%%%%%%%%%%%%%%%%%%%%%%%%%%%%%%%%%%%%%%%%%%%%%%%%%%%%%%%%%%%%%%%%%%%%%

\bigskip

\bibliographystyle{amsalpha}

\begin{thebibliography}{ABC{\etalchar{+}}16}


\bibitem[Abe16]{abe:divfree}
T.~Abe,
\emph{Divisionally free arrangements of hyperplanes},
Inventiones Math. \textbf{204(1)}, (2016), 317--346.

\bibitem[ABC{\etalchar{+}}16]{abeetall:weyl}
T.~Abe, M.~Barakat, M.~Cuntz, T.~Hoge, and H.~Terao,
\emph{The freeness of ideal subarrangements of Weyl arrangements},
JEMS,  \textbf{18} (2016), no. 6, 1339--1348.

\bibitem[AT16]{abeterao:freefiltrations}
T.~Abe and H.~Terao,
\emph{Free filtrations of affine Weyl arrangements and the ideal-Shi arrangements},
J. Algebraic Combin. \textbf{43} (2016), no. 1, 33--44. 

\bibitem[Aig79]{aigner:combinatorialtheory}
M.~Aigner, 
\emph{Combinatorial Theory}, Grundlehren der math.\ Wissenschaften
vol.\ 234 (1979).

\bibitem[BC12]{barakatcuntz:indfree} M. Barakat and M. Cuntz, \emph{Coxeter and 
    crystallographic arrangements are inductively free}, Adv. Math \textbf{229}
    (2012), 691--709.

\bibitem[BEZ90]{bjoerneredelmanziegler}
A.~Bj\"orner, P.~Edelman, and G.~Ziegler,
\emph{Hyperplane arrangements with a lattice of regions}.
Discrete Comput. Geom. \textbf{5} (1990), no. 3, 263--288. 

\bibitem[Bou68]{bourbaki:groupes} N.~Bourbaki, \emph{\'{E}l\'ements de
    math\'ematique. {G}roupes et alg\`ebres de {L}ie. {C}hapitre {IV}-{VI}},
  Actualit\'es Scientifiques et Industrielles, No. 1337, Hermann, Paris,
  1968.

\bibitem[Car94]{carrell}
J. B. Carrell, 
\emph{The Bruhat graph of a Coxeter group, a conjecture of Deodhar, and
rational smoothness of Schubert varieties}, 
Algebraic groups and their generalizations:
Classical methods (University Park, 1991), Proc. Sympos. Pure Math., \textbf{56}, part 1,
pp. 53–61, Amer. Math. Soc., Providence, RI, 1994.

\bibitem[CP00]{cellinipapi:nilpotentI}
P.~Cellini and P.~Papi, 
\emph{ad-nilpotent ideals of a Borel subalgebra}.
J. Algebra \textbf{225} (2000), no. 1, 130--141. 

\bibitem[CP02]{cellinipapi:nilpotentII}
\bysame, %P.~Cellini, P.~Papi, 
\emph{ad-nilpotent ideals of a Borel subalgebra. II}. 
J. Algebra \textbf{258} (2002), %no. 1, 
112--121.

\bibitem[Ed84]{edelman:regions}
P.H.~Edelman, 
\emph{A partial order on the regions of $\BBR^n$ 
dissected by hyperplanes}. Trans. Amer. Math. Soc. \textbf{283} (1984), no. 2, 617--631. 

\bibitem[ER94]{edelmanreiner}
P.H.~Edelman and V.~Reiner, 
\emph{Free hyperplane arrangements between $A_{n - 1}$ and $B_n$} 
Math. Z. \textbf{215} (1994), no. 3, 347--365. 

\bibitem[Gr71]{gruenbaum}
B.~Gr\"unbaum,
\emph{Arrangements of Hyperplanes}, 
Proceedings of the Second Louisiana Conference on Combinatorics, 
Graph Theory and Computing (Louisiana State Univ., Baton Rouge, La., 1971), 
pp.~41--106. 

\bibitem[HR14]{hogeroehrle:super} 
T.~Hoge and G.~R\"ohrle, 
\emph{On supersolvable reflection arrangements}, 
 Proc. AMS, \textbf{142} (2014), no. 11, 3787--3799.

\bibitem[HR15]{hogeroehrle:indfree} 
\bysame, %T.~Hoge and G.~R\"ohrle, 
\emph{On inductively free reflection arrangements}, 
J.~Reine u.~Angew.~Math. 
\textbf{701} (2015), 205--220.

\bibitem[HR16a]{hogeroehrle:factored} 
\bysame, %T.~Hoge and G.~R\"ohrle, 
\emph{Addition-Deletion Theorems for Factorizations 
of Orlik-Solomon Algebras and nice Arrangements},
European J. Combin. \textbf{55} (2016), 20--40. 

\bibitem[HR16b]{hogeroehrle:nice} 
\bysame, %T.~Hoge and G.~R\"ohrle, 
\emph{Nice Reflection Arrangements},
Electronic J. Combin. \textbf{23(2)}, (2016), \# P2.9.

\bibitem[HRS17]{hogeroehrleschauenburg:free} 
T.~Hoge, G.~R\"ohrle and A.~Schauenburg, 
\emph{Inductive and Recursive Freeness of Localizations of 
multiarrangements}, to appear. 
%\url{http://arxiv.org/abs/1501.06312}

\bibitem[Hul16]{hultman:koszul} 
A.~Hultman,
\emph{Supersolvability and the Koszul property of root ideal arrangements},
 Proc. AMS, \textbf{144} (2016), 1401--1413.

\bibitem[JP95]{jambuparis:factored}
M.~Jambu and L.~Paris,
\emph{Combinatorics of Inductively Factored Arrangements},
European J. Combin. \textbf{16} (1995), 267--292.

\bibitem[JT84]{jambuterao:free} 
M.~Jambu and H.~Terao, 
\emph{Free arrangements of hyperplanes and 
supersolvable lattices},
Adv. in Math. \textbf{52} (1984), no.~3, 248--258. 

\bibitem[Ko59]{kostant:betti}
B. Kostant, 
\emph{The principal three-dimensional subgroup 
and the Betti numbers of a complex simple Lie group}.
Amer. J. Math. \textbf{81} (1959) 973--1032.

\bibitem[Ko61]{kostant:borelweil}
\bysame, %B. Kostant, 
\emph{Lie algebra cohomology and the generalized Borel-Weil theorem}. 
Ann. of Math. (2) \textbf{74} (1961) 329--387.

\bibitem[Mac72]{macdonald:coxeter}
I. G.  Macdonald, 
\emph{The Poincar\'e series of a Coxeter group}. 
Math. Ann. \textbf{199} (1972), 161--174. 

\bibitem[MR17]{muellerroehrle:factored}
T.~M\"oller and G.~R\"ohrle,
\emph{Localizations of inductively factored arrangements},
to appear.
%\url{http://arxiv.org/abs/1602.06421}

\bibitem[OY10]{ohyoo}
S.~Oh and H.~Yoo, 
\emph{Bruhat order, rationally smooth Schubert varieties, 
and hyperplane arrangements}. 
22nd International Conference on Formal Power Series 
and Algebraic Combinatorics (FPSAC 2010), 965--972,
Discrete Math. Theor. Comput. Sci. Proc., 
AN, Assoc. Discrete Math. Theor. Comput. Sci., Nancy, 2010.

\bibitem[OST84]{orliksolomonterao:hyperplanes} 
P.~Orlik, L.~Solomon, and H.~Terao, 
\emph{Arrangements of hyperplanes and differential forms}. 
Combinatorics and algebra (Boulder, Colo., 1983), 29--65,
Contemp. Math., \textbf{34}, Amer. Math. Soc., Providence, RI, 1984. 

\bibitem[OT92]{orlikterao:arrangements} P.~Orlik and H.~Terao,
  \emph{Arrangements of hyperplanes}, Springer-Verlag, 1992.

\bibitem[Shi97]{shi:signtype} 
J.-Y.~Shi, 
\emph{The number of $\oplus$-sign types}, 
Quart. J. Math. Oxford Ser. (2) 48 (1997), no. 189, 93--105. 

\bibitem[Slo15]{slofstra:schubert}
W.~Slofstra, 
\emph{Rationally smooth Schubert varieties 
and inversion hyperplane arrangements},
Adv. Math. \textbf{285} (2015), 709--736.

\bibitem[Sol66]{solomon:chevalley}
L.~Solomon, 
\emph{The orders of the finite Chevalley groups}.
J. Algebra \textbf{3} (1966) 376--393. 

\bibitem[Som05]{sommers:ideals}
E.~Sommers, 
\emph{B-stable ideals in the nilradical of a Borel subalgebra}. 
Canad. Math. Bull. \textbf{48} (2005), no. 3, 460--472.

\bibitem[ST06]{sommerstymoczko}
E.~Sommers and J.~Tymoczko, 
\emph{Exponents for $B$-stable ideals}. 
Trans. Amer. Math. Soc. \textbf{358} (2006), no. 8, 3493--3509. 

\bibitem[Sta72]{stanley:super}
R. P. Stanley, 
\emph{Supersolvable lattices},
Algebra Universalis \textbf{2} (1972), 197--217. 

\bibitem[Sta12]{stanley:book}
\bysame, %R.~Stanley, 
\emph{Enumerative combinatorics}. Volume 1. 
Second edition. Cambridge Studies in Advanced Mathematics, 
\textbf{49}. Cambridge University Press, Cambridge, 2012.

\bibitem[Sta99]{stanley:book2}
\bysame, %R.~Stanley, 
\emph{Enumerative combinatorics}. Volume 2. 
Second edition. Cambridge Studies in Advanced Mathematics, 
\textbf{62}. Cambridge University Press, Cambridge, 1999.

\bibitem[Ste59]{steinberg:reflectiongroups}
R. Steinberg, 
\emph{Finite reflection groups}. Trans. Amer. Math. Soc., \textbf{91}
(1959), 493--504.

\bibitem[Ste60]{steinberg:invariants}
\bysame, %R. Steinberg, 
\emph{Invariants of finite reflection groups},
Canad. J. Math. \textbf{12}, (1960), 616--618.

\bibitem[Ter80]{terao:freeI} 
H.~Terao, 
\emph{Arrangements of hyperplanes and
    their freeness I}, J. Fac. Sci.  Univ. Tokyo \textbf{27} (1980), 293--320.

\bibitem[Ter81]{terao:freefactors}
\bysame,
\emph{Generalized exponents of a free arrangement of hyperplanes and 
Shepherd-Todd-Brieskorn formula}, Invent. Math. \textbf{63} no.\ 1 (1981) 159--179.

\bibitem[Ter92]{terao:factored}
\bysame,
\emph{Factorizations of the Orlik-Solomon Algebras},
Adv. in Math. \textbf{92}, (1992), 45--53.

\bibitem[Tym07]{tymoczko}
J.~Tymoczko, 
\emph{Paving Hessenberg varieties by affines},
Sel. math., New ser. \textbf{13}, (2007), 353--367.

\end{thebibliography}

\newcommand{\etalchar}[1]{$^{#1}$}
\providecommand{\bysame}{\leavevmode\hbox to3em{\hrulefill}\thinspace}
\providecommand{\MR}{\relax\ifhmode\unskip\space\fi MR }
\providecommand{\MRhref}[2]{%
  \href{http://www.ams.org/mathscinet-getitem?mr=#1}{#2} }
\providecommand{\href}[2]{#2}

%%%%%%%%%%%%%%%%%%%%%%%%%%%%%%%%%%%%%%%%%%%%%%%%%%%%%%%%%%%%%%%%%%%%%%
%%%%%%%%%%%%%%%%%%%%%%%%%%%%%%%%%%%%%%%%%%%%%%%%%%%%%%%%%%%%%%%%%%%%%%

\end{document}